\documentclass{amsart}

\usepackage[mathscr]{eucal}
\usepackage{amssymb}
\usepackage{pifont}
\usepackage{graphicx}
\usepackage{psfrag} 
\usepackage{epstopdf} 
\usepackage[usenames,dvipsnames]{color}
\usepackage[normalem]{ulem}
\usepackage{amsthm}
\usepackage{bbm}
\usepackage{enumerate}
\usepackage{array}
\usepackage{amsmath}

\usepackage{hyperref}
\hypersetup{colorlinks=true,linkcolor={Brown},citecolor={Brown},urlcolor={Brown}}

\usepackage{cleveref}

\numberwithin{equation}{section}
\makeindex

\usepackage[all]{xy}
\CompileMatrices

\newdir{ >}{{}*!/-10pt/\dir{>}}

\swapnumbers 

\newtheorem{Thm}[equation]{Theorem}
\newtheorem{Prop}[equation]{Proposition}
\newtheorem{Lem}[equation]{Lemma}
\newtheorem{Cor}[equation]{Corollary}

\theoremstyle{remark}
\newtheorem{Rem}[equation]{Remark}
\newtheorem*{Rem*}{Remark}
\newtheorem{Def}[equation]{Definition}
\newtheorem{Ter}[equation]{Terminology}
\newtheorem{Not}[equation]{Notation}
\newtheorem{Exa}[equation]{Example}

\newtheorem{Cons}[equation]{Construction}

\newtheorem{parag}[equation]{}
\newtheorem{Warning}[equation]{Warning}

\theoremstyle{definition}
\newtheorem*{Ack}{Acknowledgements}

\newcommand{\nc}{\newcommand}
\nc{\dmo}{\DeclareMathOperator}

\dmo{\Ab}{Ab}
\dmo{\AbMon}{AbMon}
\dmo{\Aut}{Aut}
\dmo{\bicMack}{\biMack_{\mathsf{ic}}} 
\dmo{\biMack}{\mathsf{Mack}} 
\dmo{\Ch}{Ch}
\dmo{\CoInd}{CoInd}
\dmo{\Der}{D}
\dmo{\End}{End}
\dmo{\Fun}{\mathrm{Fun}} 
\dmo{\Hom}{Hom}
\dmo{\Ho}{Ho}
\dmo{\img}{im}
\dmo{\Img}{Im}
\dmo{\incl}{incl}
\dmo{\Ind}{Ind}
\dmo{\Infl}{Inf}
\dmo{\Defl}{Def}
\dmo{\Iso}{Iso}
\dmo{\Inj}{Inj} 
\dmo{\Ker}{Ker}
\dmo{\Mackey}{Mack} 
\dmo{\Map}{Map}%
\dmo{\Mod}{Mod}
\dmo{\Mor}{Mor}%
\dmo{\Obj}{Obj}
\dmo{\Proj}{Proj} 
\dmo{\pr}{pr}
\dmo{\PsFunJJ}{\PsFun_{\JJ_!}^{\JJ^\prime\textrm{\!-}\mathsf{oplax}}}
\dmo{\PsFunJop}{\PsFun_{{{\JJ}_{{}_{*}}}}}
\dmo{\PsFunJ}{\PsFun_{\JJ_!}}
\dmo{\PsFunoplax}{\PsFun^{\mathsf{oplax}}}
\dmo{\PsFun}{\mathsf{PsFun}} 
\dmo{\Res}{Res}
\dmo{\SH}{SH}
\dmo{\Spanname}{{\sf Span}}
\dmo{\Stab}{Stab}
\dmo{\twoFun}{2\mathsf{Fun}}

\nc\noloc{\nobreak\mspace{6mu plus 1mu}{:}\nonscript\mkern-\thinmuskip\mathpunct{}\mspace{2mu}}
\nc{\ababs}{{\sl ab absurdo}}
\nc{\Add}{\mathsf{Add}}
\nc{\ADD}{\mathsf{ADD}}
\nc{\adhoc}{{\sl ad hoc}}
\nc{\adjto}{\rightleftarrows}
\nc{\adj}{\dashv\,}
\nc{\afortiori}{{\sl a fortiori}}
\nc{\aka}{{a.\,k.\,a.}\ }
\nc{\all}{\mathsf{all}}
\nc{\apriori}{{\sl a priori}}
\nc{\ass}{\mathrm{ass}} 
\nc{\bbA}{\mathbb{A}}
\nc{\bbB}{\mathbb{B}}
\nc{\bbC}{\mathbb{C}}
\nc{\bbD}{\mathbb{D}}
\nc{\bbF}{\mathbb{F}}
\nc{\bbI}{\mathbb{I}}
\nc{\bbM}{\mathbb{M}}
\nc{\bbN}{\mathbb{N}}
\nc{\bbP}{\mathbb{P}}
\nc{\bbQ}{\mathbb{Q}}
\nc{\bbR}{\mathbb{R}}
\nc{\bbZ}{\mathbb{Z}}
\nc{\bs}{\backslash}
\nc{\BurnG}{\cat{A}(G)}
\nc{\cat}[1]{\mathcal{#1}}
\nc{\Cat}{\mathsf{Cat}}
\nc{\CAT}{\mathsf{CAT}}
\nc{\cf}{{\sl cf.}\ }
\nc{\Cf}{{\sl Cf.}\ }
\nc{\colim}{\mathop{\mathrm{colim}}}
\nc{\costar}{**}
\nc{\co}{{\mathrm{co}}}
\nc{\DD}{\cat{D}}
\nc{\Displ}{\displaystyle}
\nc{\doublequot}[3]{#1\backslash #2/#3}
\nc{\Ecell}{\rotatebox[origin=c]{90}{$\Downarrow$}} 
\nc{\eg}{{\sl e.g.}\ } 
\nc{\Eg}{{\sl E.g.}\ } 
\nc{\eps}{\varepsilon}
\nc{\equalby}[1]{\overset{\textrm{#1}}{=}}
\nc{\exact}{\mathsf{ex}}
\nc{\faithful}{\mathsf{faithful}}
\nc{\faith}{\mathsf{faithf}}
\nc{\final}{\textrm{\scriptsize{\ding{93}}}} 
\nc{\Funadd}{\Fun_{\amalg}}
\nc{\Funplus}{\Fun_{+}}
\nc{\fun}{\mathrm{fun}} 
\nc{\GG}{\mathbb{G}}
\nc{\gpdG}{{\groupoidf_{\!\smallslash\!G}}} 
\nc{\gpdGfuz}{{\groupoid{}^{\smallfaithful,\mathsf{fus}}_{\!\smallslash\!G}}}
\nc{\fuz}{\mathsf{fus}}
\nc{\ssetfuz}{\sset^{\smallfused}} 
\nc{\gpd}{\groupoid}
\nc{\GinG}{{\groupoidf_{G}}}
\nc{\gps}{\mathsf{groups}} 
\nc{\groconn}{\groupoid_{\mathsf{conn}}}
\nc{\groupoidf}{\groupoid{}^{\smallfaithful}}
\nc{\groupoid}{\mathsf{gpd}}
\nc{\group}{\mathsf{gr}} 
\nc{\Gsets}{G\mathsf{-sets}}
\nc{\HGfK}{\doublequot{H}{G}{f(K)}}
\nc{\HGK}{\doublequot HGK}
\nc{\Homcat}[1]{\Hom_{\cat #1}}
\nc{\hooklongleftarrow}{\longleftarrow\joinrel\rhook}
\nc{\hooklongrightarrow}{\lhook\joinrel\longrightarrow}
\nc{\hook}{\hookrightarrow}
\nc{\Hsets}{H\mathsf{-sets}}
\nc{\ICAdd}{\Add_{\mathsf{ic}}}
\nc{\ICADD}{\ADD_{\mathsf{ic}}}
\nc{\Idcat}[1]{\Id_{\cat{#1}}}
\nc{\id}{\mathrm{id}}
\nc{\Id}{\mathrm{Id}}
\nc{\ie}{{\sl i.e.}\ }
\dmo{\Image}{\mathsf{Im}}
\nc{\into}{\mathop{\rightarrowtail}}
\nc{\inv}{^{-1}}
\nc{\Iout}[1]{\Ivo{\sout{#1}}}
\nc{\isocell}[1]{\undersett{ #1}{\overset{\sim}{\Ecell}}} 
\nc{\Isocell}[1]{\undersett{ #1}{\overset{\sim}{\Longrightarrow}}}
\nc{\isoEcell}{\overset{\sim}{\Rightarrow}} 
\nc{\isotoo}{\stackrel{\sim}\longrightarrow}
\nc{\isoto}{\buildrel \sim\over\to}
\nc{\Ivo}[1]{{\color{OliveGreen}#1}}
\nc{\JJ}{\mathbb{J}}
\nc{\kk}{\Bbbk}
\nc{\KK}{\mathrm{KK}}
\nc{\leps}{{}^{\ell}\eps}
\nc{\leta}{{}^{\ell}\eta}
\nc{\loccit}{{\sl loc.\ cit.}}
\nc{\lotoo}[1]{\overset{#1}{\,\longleftarrow\,}}
\nc{\loto}[1]{\overset{#1}{\leftarrow}}
\nc{\lto}{\leftarrow}
\nc{\lun}{\mathrm{lun}} 
\nc{\Mackintro}[1]{(Mack\,\ref{Mack-#1-intro})}
\nc{\Mack}[1]{(Mack\,\ref{Mack-#1})}
\nc{\Mid}{\,\big|\,}
\nc{\MMod}{\text{-}\!\Mod}%
\nc{\MM}{\cat{M}}
\nc{\Muniv}{\cat{M}_{\mathsf{univ}}}
\nc{\Ncell}{\rotatebox[origin=c]{0}{$\Uparrow$}} 
\nc{\NEcell}{\rotatebox[origin=c]{135}{$\Downarrow$}} 
\nc{\NN}{\cat{N}}
\nc{\NWcell}{\rotatebox[origin=c]{-135}{$\Downarrow$}} 
\nc{\oEcell}[1]{\overset{\scriptstyle #1}{\Ecell}} 
\nc{\oWcell}[1]{\overset{\scriptstyle #1}{\Wcell}} 
\nc{\ointo}[1]{\overset{#1}{\rightarrowtail}}
\nc{\olto}[1]{\overset{#1}\lto}
\nc{\onto}{\mathop{\twoheadrightarrow}}
\nc{\op}{{\mathrm{op}}}
\nc{\otoo}[1]{\overset{#1}{\,\longrightarrow\,}}
\nc{\oto}[1]{\overset{#1}\to}
\nc{\Paul}[1]{{\color{Blue}#1}}
\nc{\pih}[1]{\tau_{1}#1}
\nc{\Pout}[1]{\Paul{\sout{#1}}}
\nc{\PsFunJindex}{\PsFun_{{\JJ_!}} \ \ {{\JJ}_{!}}\textrm{-strong pseudo-functors}}
\nc{\qquadtext}[1]{\qquad\textrm{#1}\qquad}
\nc{\quadtext}[1]{\quad\textrm{#1}\quad}
\nc{\ra}{\rightarrow}
\nc{\Real}{\mathrm{Re}}
\nc{\reps}{{}^{r\!}\eps}
\nc{\restr}[1]{{|_{\scriptstyle #1}}}
\nc{\reta}{{}^{r\!}\eta}
\nc{\run}{\mathrm{run}} 
\nc{\Sad}{\mathsf{Sad}}
\nc{\SAD}{\mathsf{SAD}}
\nc{\sbull}{{\scriptscriptstyle\bullet}}
\nc{\Scell}{\rotatebox[origin=c]{0}{$\Downarrow$}} 
\nc{\SEcell}{\rotatebox[origin=c]{45}{$\Downarrow$}} 
\nc{\SET}[2]{\big\{\,#1\Mid#2\,\big\}}
\nc{\set}{\mathsf{set}} 
\nc{\Set}{\mathsf{Set}}
\nc{\smallfaithful}{\mathsf{f}}
\nc{\smallfused}{\mathsf{fus}}
\nc{\ssetfused}{\textrm{-}\underline{\set}} 
\nc{\smallslash}{{}^{\scriptscriptstyle/}}
\nc{\smat}[1]{\left(\begin{smallmatrix} #1 \end{smallmatrix}\right)}
\nc{\spanG}{{\widehat{\mathsf{gp}\,\,}\!\!\mathsf{d}}{}^\smallfaithful_{\!{}^{\scriptscriptstyle/}\!G}}
\nc{\Spanhat}{\textrm{\sf S}\widehat{\textrm{\sf pan}}} %
\nc{\Span}{\Spanname}
\nc{\spancat}{\mathrm{Sp}} 
\nc{\spank}{\spancat_{\kk}}
\nc{\biset}{\mathsf{biset}} 
\nc{\bisetcat}{\mathrm{Bis}} 
\nc{\bisetcatbif}{\mathrm{Bis}^\mathrm{bif}_\kk\!} 
\nc{\bisetcatrf}{\mathrm{Bis}^\mathrm{rf}_\kk\!} 
\nc{\bisetfun}{\mathrm{BisF}} 
\nc{\bifree}{\mathsf{bif}}
\nc{\rfree}{\mathsf{rf}}
\nc{\lfree}{\mathsf{lf}}
\nc{\sset}{\textrm{-}\set}
\nc{\str}{\mathsf{str}}
\nc{\SWcell}{\rotatebox[origin=c]{-45}{$\Downarrow$}} 
\nc{\too}{\mathop{\longrightarrow}\limits}
\nc{\tSpan}{\pih{\Spanname}}
\nc{\undersett}[1]{\underset{\scriptstyle #1}}
\nc{\un}{\mathrm{un}} 
\nc{\vcorrect}[1]{{\vphantom{\vbox to #1em{}}}}
\nc{\Wcell}{\rotatebox[origin=c]{90}{$\Uparrow$}} 
\nc{\what}[1]{\widehat{\cat{#1}}}
\nc{\xra}{\xrightarrow}

\begin{document}


\title[Axiomatic representation theory via groupoids]{Axiomatic representation theory of finite groups by way of groupoids
}
\author{Ivo Dell'Ambrogio}
\date{\today}

\address{
\noindent  Univ.\ Lille, CNRS, UMR 8524 - Laboratoire Paul Painlev\'e, F-59000 Lille, France 
}
\email{ivo.dell-ambrogio@univ-lille.fr}
\urladdr{http://math.univ-lille1.fr/$\sim$dellambr}

\begin{abstract} \normalsize
We survey several notions of Mackey functors and  biset functors  found in the literature and prove some old and new theorems comparing them. While little here will surprise the experts, we draw a conceptual and unified picture by making systematic use of finite groupoids. This provides a `road map' for the various approaches to the axiomatic representation theory of finite groups,  as well as some details which are hard to find in writing.
\end{abstract}

\thanks{Author partially supported by Project ANR ChroK (ANR-16-CE40-0003) and Labex CEMPI (ANR-11-LABX-0007-01).}


\maketitle


\tableofcontents

\section{Introduction and results}
\label{sec:intro}%

This is a survey of several variants of Mackey functors and biset functors for finite groups appearing in the literature. (Beware: we survey \emph{abstract formalisms}; the reader interested in \emph{concrete examples} is referred to~\cite{Webb00}.) Our goal is to show that it becomes quite easy to relate these variants to one another and to rigorously prove comparison theorems, provided one embraces \emph{finite groupoids} (categories with finitely many arrows all of which are invertible) and \emph{2-categories} (categories equipped with 2-morphisms, \ie arrows between arrows). The main reason for using finite groupoids is because they include all finite groups as well as all finite $G$-sets for each group~$G$; the main reason for using the language of 2-categories is in order to exploit the fact that finite groupoids, together with functors and natural transformations, form a very nicely behaved 2-category.

Let us begin with a quick review of what is often referred to as ``axiomatic representation theory''. Roughly speaking, both Mackey functors and biset functors provide ways of encoding the various homomorphisms that arise in  the representation theory of finite groups when one allows the group to vary. One typically encounters \emph{restriction} maps and \emph{induction} (also called \emph{transfer} or \emph{trace}) maps associated with inclusions $H\hookrightarrow G$ of subgroups, and possibly also \emph{inflation} and \emph{deflation} maps, associated with quotients $G\twoheadrightarrow G/N$ by normal subgroups. There may also be \emph{isomorphism} maps coming from abstract isomorphisms of groups $G\cong G'$, or at least the special case of \emph{conjugation} maps induced by conjugations $G\cong gGg^{-1}$ by an element~$g$ of some fixed `ambient' group of which $G$ is a subgroup. These families of maps interact by various (long) lists of basic relations, which are then promoted to the role of axioms of a formal algebraic theory.

Note that the above families of homomorphisms come in pairs of opposite variance: restriction/induction and inflation/deflation, with isomorphisms and conjugations having both variances since they are invertible. A classical idea, due to Lindner \cite{Lindner76}, is to simultaneously encode both variances in some category of `spans', \ie diagrams of the form~$X \gets S \to Y$ where the `wrong-way' maps will induce the `wrong-way' functoriality (variations on this idea abound in mathematics, see \eg the many uses of \emph{correspondences} in geometry). Another way to impose this symmetry is by using bisets, as proposed by Bouc~\cite{Bouc10}. Then a Mackey functor, resp.\ a biset functor, can be simply defined to be a linear representation of the category of spans, resp.\ of bisets.

\begin{Warning} \label{Warning:Mackey}
Our usage of `Mackey functors' for functors defined on any kind of span category originates in \cite{Lindner76} and is now quite widespread. However, it is at odds with the tradition in representation theory, where the qualifier `Mackey' is typically reserved for functors with restriction and induction maps but no inflations or deflations (if they have inflation maps, for instance, they will be called `inflation functors', without the qualifier `Mackey'). The two uses seem hard to reconcile, in particular with respect to the global variants surveyed below.
\end{Warning}

Let us first review Mackey functors, using groupoids. 
%

\bigbreak 
\noindent\textbf{Mackey functors}\\  \nopagebreak 
\vspace{-.2cm}

\noindent The unifying point of view we adopt here is that a \emph{Mackey functor} $M$ should be defined to be an abelian-groups-valued (or later, more generally, taking values in modules over some commutative ring~$\kk$) additive functor on the category of spans formed in a suitable 2-category of finite groupoids (see \Cref{sec:Mackey-for-2-1-category}). Let us insist straight away that the additivity condition, $M(G_1 \sqcup G_2)\cong M(G_1) \oplus M(G_2)$, implies that the \emph{data} of a Mackey functor can always be reduced to what happens to indecomposable groupoids, \ie good old finite groups. Nevertheless, it is convenient to work with groupoids because they make definitions more conceptual and results easier to see and to prove.

We allow two parameters in this definition of Mackey functor. Firstly, the above-mentioned `suitable' 2-category of groupoids, denoted below by~$\GG$, can be adjusted as needed: it will typically be either a (2-full) sub-2-category of the 2-category of all finite groupoids, or a comma 2-category over a fixed group(oid).
Secondly, we further choose a distinguished (wide) sub-2-category $\JJ\subseteq \GG$ which determines which functors of groupoids are allowed to induce `wrong way' maps (\eg inductions or deflations). By choosing the pair $(\GG;\JJ)$ adequately, the resulting notion of Mackey functor can be specialized to those in common use. By way of illustration, we will explicitly consider five of them (leaving further variations to the interested reader):
\begin{enumerate}[(1)]
\item \label{it:Mackey-fixedG}
The original \emph{Mackey functors for a fixed group~$G$} \cite{Green71} \cite{Dress73}. They are equipped with: restriction maps, induction maps and conjugations in~$G$. They appear all over equivariant mathematics, perhaps most notably in equivariant stable homotopy theory, as the algebraic structure with which the homotopy groups of `genuine' $G$-spectra are naturally endowed (\cite{LewisMaySteinbergerMcClure86} \cite{Carlsson92}).  
\item \label{it:Mackey-globaltot}
\emph{Global Mackey functors} in the sense of \cite{Ganter13pp} and \cite{Nakaoka16} \cite{Nakaoka16a}. These are defined on all finite group(oid)s and have maps of each kind (that is: restrictions, inductions, all isomorphisms, inflations and deflations).
\item \label{it:Mackey-globalnoinfl}
Global Mackey functors as above, but without deflation maps. These have been given various names: \emph{functors with regular Mackey structure} \cite{Symonds91}, \emph{inflation functors} \cite{Webb93}, and \emph{global $(\emptyset,\infty)$-Mackey functors} \cite{Lewis99}. They appear, for instance, as the natural algebraic invariant of Schwede's global equivariant spectra (see \cite[Thm.\,4.2.6]{Schwede18} and the discussion after it).
\item \label{it:Mackey-globalnoinfldefl}
Global Mackey functors as above, but without inflation and deflation maps. These are simply called \emph{global Mackey functors} in \cite{Webb93} and \cite{Bouc10} (\cf \Cref{Warning:Mackey}).
\item \label{it:Mackey-fused}
The \emph{fused Mackey functors for~$G$} of \cite{Bouc15}, also called \emph{conjugation invariant Mackey functors} in~\cite{HambletonTaylorWilliams10}, namely those Mackey functors for~$G$ as in~\eqref{it:Mackey-fixedG} which admit a reformulation as a kind of biset functors (\cf \Cref{Rem:biset(G)}).
\end{enumerate}
As we will prove, the above five types of Mackey functors can be obtained by specializing our general definition to the following choices of the parameters $(\GG;\JJ)$, respectively:
\begin{enumerate}
\item[For \!\eqref{it:Mackey-fixedG}] : use $(\GG;\JJ)=(\gpdG;\all)$, where $\gpdG$ is the comma 2-category of groupoids faithfully embedded in~$G$ and $\JJ=\all$ just means that $\JJ=\GG$.  See \Cref{sec:G-Mackey}.
\item[For \!\eqref{it:Mackey-globaltot}] : use $(\GG;\JJ)=(\gpd;\all)$, where we take the whole 2-category $\gpd$ of all finite groupoids and functors between them and where we allow the formation of all spans. See \Cref{Def:Mackeyfun-global}.
\item[For \!\eqref{it:Mackey-globalnoinfl}] : use $(\GG;\JJ)=(\gpd;\groupoidf)$, where we consider the whole category of groupoids but we only allow spans $X\gets S\to Y$ whose right leg $S\to Y$ is a faithful functor. See \Cref{Exa:gpd-gpdf}.
\item[For \!\eqref{it:Mackey-globalnoinfldefl}] : use $(\GG;\JJ)=(\groupoidf;\all)$, where both legs of all spans must be faithful functors. See \Cref{Exa:gpdf}.
\item[For \!\eqref{it:Mackey-fused}] : use $(\GG;\JJ)=(\gpdGfuz;\all)$, where $\gpdGfuz$ is the `fused' variant of the comma 2-category of groupoids faithfully embedded in~$G$. See \Cref{Def:gpdG-fused}.
\end{enumerate}

As the reader may guess, the above 2-categories are all related by evident inclusion and forgetful 2-functors. We will exploit this fact in order to easily establish comparison results.

In order to provide a uniform and conceptual construction of the span category for all of the above examples (and many more), we introduce the general notion of a \emph{spannable pair} $(\GG;\JJ)$ (see \Cref{Def:spannable}). A spannable pair consists of an extensive (2,1)-category $\cat E$ equipped with a suitably closed 2-subcategory~$\JJ$ and sufficiently many \emph{Mackey squares} (\ie pseudo-pullbacks of groupoids). 
This abstract approach is developed in~\Cref{sec:Mackey-for-2-1-category}. 
In \Cref{sec:presentation}, we look in full details at the span category for the basic example~\eqref{it:Mackey-globaltot}, in order to dispel the (possibly intimidating) categorical abstractions of the general definition by reducing it to some classical combinatorics. In particular, we describe a presentation of the linear category of spans of groupoids (see \Cref{Thm:pres-Span}).

Of course, we also need to explain how to connect the above definitions with the more familiar ones found in the literature. We will now briefly explain how to do this, beginning with \eqref{it:Mackey-fixedG} and~\eqref{it:Mackey-fused} and the associated comma 2-categories.
%
%

\bigbreak
\noindent\textbf{From $G$-sets to groupoids: the transport groupoid}\\  \nopagebreak 
\vspace{-.2cm}

\noindent 
Mackey functors for a fixed group $G$, as in type~\eqref{it:Mackey-fixedG} above, are typically expressed in terms of $G$-sets.
The key tool for comparing Mackey functors for a fixed~$G$ with global Mackey functors is the \emph{transport groupoid functor} $G\ltimes-$, which sends a $G$-set $X$ to its \emph{transport groupoid} (a.k.a.\ \emph{action groupoid}, \emph{homotopy quotient} or \emph{Grothendieck construction}) $G\ltimes X$. The latter groupoid is canonically equipped with a faithful functor $G\ltimes X \rightarrowtail G$, which turns the transport groupoid construction into a functor 
\[ 
G\ltimes - \colon G\sset\longrightarrow \gpdG 
\]
from the category of $G$-sets into the comma 2-category of groupoids `faithfully embedded' in~$G$ (\Cref{Def:gpdG}). This functor is a nice inclusion, in fact it is a biequivalence (an equivalence of 2-categories). As a consequence, Mackey functors for $G$, which can be defined to be linear representations of the category of spans in $G$-sets, turn out to be equivalent to representations of the category of spans in~$\gpdG$. Therefore they are the result of specializing our general notion of Mackey functors to the pair $(\GG;\JJ)=(\gpdG;\all)$ (see \Cref{Cor:equivalence-MackeyG}). 
All of this is already contained in \cite[\S\,B.1]{BalmerDellAmbrogio18pp} but is briefly recalled at the beginning of \Cref{sec:G-Mackey} for the reader's convenience.

The `$G$-local' and the `global' settings are now compared by the forgetful 2-functor $\gpdG\to \gpd$ which simply forgets the embedding $H\rightarrowtail G$ into~$G$. This 2-functor is not 2-full, because `being over~$G$' puts a constraint on the natural isomorphisms between (faithful) functors that can be used in the comma 2-category, while in $\gpd$ we can use all of them. 
If we pull back these extra 2-cells and add them to~$\gpdG$, we obtain a variant of the comma 2-category which we denote~$\gpdGfuz$ and call the 2-category of \emph{fused groupoids embedded into~$G$} (see \Cref{Def:gpdG-fused}). The 2-cells can be pulled further back onto the category of $G$-sets, which results in a 2-category $G\ssetfuz$ consisting of finite $G$-sets, $G$-maps and \emph{twisting maps} relating parallel $G$-maps (see \Cref{Def:G-set-fuz}). If we truncate the 2-category $G\ssetfuz$, the result is precisely Bouc's category $G\ssetfused$ of fused Mackey functors for~$G$. By definition, fused Mackey functors (type \eqref{it:Mackey-fused} above) are representations of spans in~$G\ssetfused$. 

It follows that fused Mackey functors can be recovered as the Mackey functors for the 2-category $\GG=\JJ = \gpdGfuz \simeq G\ssetfuz$ (see \Cref{Cor:fusedMackeyfun_vs_gpdGfuz-Mackeyfun}). The above arguments make it also easy to identify fused Mackey functors as those Mackey functors $M$ for $G$ which are \emph{conjugation invariant}, \ie such that the centralizer $C_G(H)$ acts trivially on the value $M(H)$ for every $H\leq G$.

This is all explained in \Cref{sec:G-Mackey}.

\bigbreak 
\noindent\textbf{Biset functors}\\  \nopagebreak 
\vspace{-.2cm}

\noindent
As already mentioned, an alternative way to force symmetry on finite group(oid)s is to use bisets rather than spans. By definition, a (finite) \emph{biset}  (also called a \emph{profunctor} or \emph{bimodule})  $U\colon H\to G$ between two groupoids is a functor $U\colon H^{\op}\times G\to \set$ to (finite) sets\footnote{Here the symmetry appears because every $(G,H)$-biset can be turned into an $(H,G)$-biset by precomposing with the inverse-arrows isomorphisms $(-)^{-1}\colon G^\op\overset{\sim}{\to} G$ and $H\overset{\sim}{\to} H^\op$.}. Taken up to isomorphism and composed by tensor products (coends), bisets are the morphisms of a category with arbitrary finite direct sums. 
Similarly to Mackey functors, and following Bouc~\cite{Bouc10}, we define here a \emph{biset functor} to be an additive functor on some suitable (sub)category of bisets of groupoids. Here too additivity allows us to reduce everything to finite groups, hence in particular our definition of biset functors is equivalent to Bouc's definition, which only uses groups; but again, we want to keep all groupoids as they provide direct sums and allow us to define the realisation of spans  (see below) in a natural way. 

Just as with spans, also with bisets there are parameters we can twiddle: we can restrict the allowed class of objects (groupoids), or the allowed class of morphisms (bisets).  We leave variations of the former kind to the interested reader. For the latter, we will study the following common three choices:
\begin{enumerate}[\rm(1)]
\item[$(2)'$] Allow all bisets.
\item[$(3)'$] Only allow right-free bisets: the resulting biset functors (called \emph{inflation functors} in \cite{Bouc10}) will have all types of maps except for deflations.
\item[$(4)'$] Only allow bi-free bisets: the resulting biset functors (called \emph{global Mackey functors} in \cite{Bouc10}) will have neither deflations nor inflations.
\end{enumerate} 
Note that Webb \cite[\S8]{Webb00} provides a more combinatorial definition of biset functors, directly in terms of the maps induced by morphisms of groups and their relations (as in the classical definition of Mackey functor for~$G$), without mentioning bisets, and under the name \emph{globally defined Mackey functors} (see \Cref{Rem:Webb's-variations}.)

This is all explained in~\Cref{sec:bisets}.
%

\bigbreak 
\noindent\textbf{The realization of spans as bisets}\\  \nopagebreak 
\vspace{-.2cm}

\noindent
The key tool for comparing Mackey functors and biset functors is the \emph{realization functor $\cat R$}  from spans to bisets 
\[
\cat R\colon \Span(\gpd) \longrightarrow \biset
\]
which sends a groupoid to itself and `realizes' each (abstract) span of functors as a (concrete) biset. This construction, which already exists as a pseudo-functor between the \emph{bi}categories of spans and bisets, was conjectured by Hoffnung \cite[Claim~13]{Hoffnung12}, was foreshadowed by Nakaoka \cite{Nakaoka16} \cite{Nakaoka16a}, and was studied in more details in Huglo's thesis~\cite{Huglo19pp}. We recall its relevant features in~\Cref{Thm:Huglo_bifunctor}. 
The way spans are realized as bisets is actually rather obvious and has been known to category theorists for a long time. What is apparently less known, but crucial for us, is the \emph{(pseudo-) functoriality} of the construction, which holds provided one composes spans using Mackey squares, as we do (see \S\,\ref{parag:isocomma}).

As a consequence of the mere existence of the realization functor~$\cat R$, we can take any biset functor and pre-compose it with~$\cat R$ in order to obtain a global Mackey functor.
By matching the parameter choices for spans and bisets, we then obtain various comparison results involving \eg the global Mackey functors of kind (2), (3) or (4) as above. 

What is interesting here is that, \emph{unless deflation maps are included}, this comparison yields an \emph{equivalence} between the corresponding categories of global Mackey functors and of biset functors   (that is, we have $(3)=(3)'$ and $(4)=(4)'$ but $(2)\neq(2)'$); see \Cref{Cor:no-deflations}. 
A version of this result was proved by Miller \cite{Miller17}; \cf \Cref{Rem:Miller}.
If deflations are included in the package, then the two notions diverge and precomposition with $\cat R$ only yields an inclusion of biset functors as a full reflective subcategory of the corresponding category of global Mackey functors. The image of the latter inclusion consists precisely of those global Mackey functors that satisfy an extra identity called the \emph{deflativity relation}.
 This applies \eg to Mackey and biset functors with all maps, as in kind~(2) above; see \Cref{Cor:Nakaoka-embedding}.
This result is due independently to Ganter \cite[App.\,A]{Ganter13pp} (whose proof uses Webb's description of biset functors) and \cite{Nakaoka16} (whose proof uses, instead of groupoids, a biequivalent 2-category $\mathbb S$ of `variable group actions').
%
%
\bigbreak 
\noindent\textbf{A road map of the formal representation theory of finite groups}\\  \nopagebreak 
\vspace{-.2cm}

\noindent
To sum up, let us collect all of the above in a single picture:
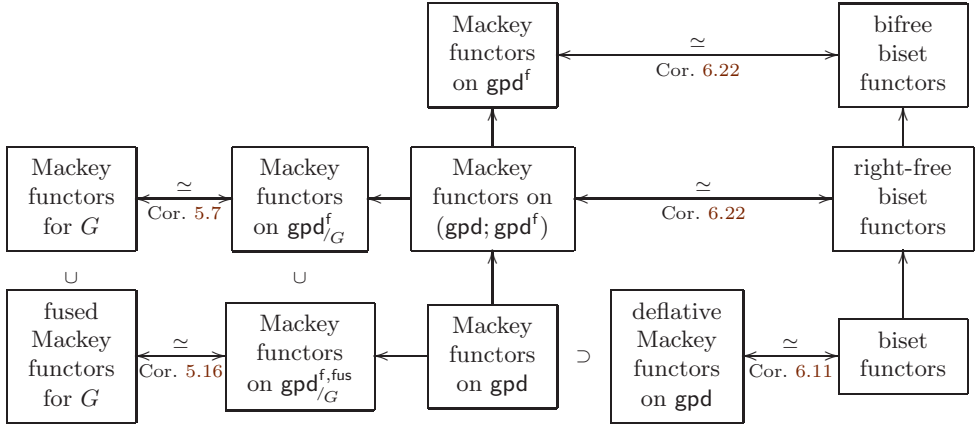
\[
\xymatrix@C=1.4em@R=1.5em{
&& &
*+[F]{\small
 \begin{array}{c}
  \text{Mackey} \\
  \text{functors} \\
  \text{on $\groupoidf$}
 \end{array}
}
\ar@{<-}[d] \ar@{<->}[rrr]^-{\simeq}_-{\textrm{Cor.\ \ref{Cor:no-deflations}}}
& && *+[F]{\small
 \begin{array}{c}
  \text{bifree} \\
  \text{biset} \\
  \text{functors}
 \end{array}
}
\ar@{<-}[d]
\\
*+[F]{\small
 \begin{array}{c}
  \text{Mackey} \\
  \text{functors} \\
  \text{for $G$}
 \end{array}
}
\ar@{<->}[rr]^-{\simeq}_-{\textrm{Cor.\ \ref{Cor:equivalence-MackeyG}}}
 &&
*+[F]{\small
 \begin{array}{c}
  \text{Mackey} \\
  \text{functors}\\
  \text{on $\gpdG$}
 \end{array}
}
 \ar@{<-}[r]
 &
*+[F]{\small
 \begin{array}{c}
  \text{Mackey} \\
  \text{functors on} \\
  \text{$(\groupoid;\groupoidf)$}
 \end{array}
}
 \ar@{<-}[d] \ar@{<->}[rrr]^-{\simeq}_-{\textrm{Cor.\ \ref{Cor:no-deflations}}}
  & &&
*+[F]{\small
 \begin{array}{c}
  \text{right-free} \\
  \text{biset} \\
  \text{functors}
 \end{array}
}
\ar@{<-}[d]
\\
*+[F]{\small
 \begin{array}{c}
  \text{fused} \\
  \text{Mackey} \\
  \text{functors} \\
  \text{for $G$}
 \end{array}
}
\ar@{}[u]|{\cup}
&&  *+[F]{\small
 \begin{array}{c}
  \text{Mackey} \\
  \text{functors} \\
  \text{on $\gpdGfuz$}
 \end{array}
}
 \ar@{<->}[ll]_-{\simeq}^-{\textrm{Cor.\ \ref{Cor:fusedMackeyfun_vs_gpdGfuz-Mackeyfun}}} \ar@{}[u]|{\cup}
& *+[F]{\small
 \begin{array}{c}
  \text{Mackey} \\
  \text{functors} \\
  \text{on $\groupoid$}
 \end{array}
}
\ar@{->}[l]
 &
*+[F]{\small
 \begin{array}{c}
  \text{deflative} \\
  \text{Mackey}\\
  \text{functors}\\
  \text{on $\groupoid$}
 \end{array}
}
\ar@{}[l]|-{\supset}
 &&
*+[F]{\small
 \begin{array}{c}
  \text{biset} \\
  \text{functors}
 \end{array}
}
\ar@{<->}[ll]_-{\simeq}^-{\textrm{Cor.\ \ref{Cor:Nakaoka-embedding}}}
}
\]
Each box is an abelian category of some sort of Mackey or biset functors, where we have indicated the parameter pair $(\GG;\JJ)$ where necessary, with $(\GG;\GG)=:\GG$ for short. The arrows represent exact functors, with equivalences marked by ``$\simeq$'' and fully faithful inclusions by ``$\subset$''.
We recall that $\groupoidf$ is the 2-category of finite groupoids with only faithful functors, $\gpdG$ is the comma 2-category of groupoids faithfully embedded in~$G$ (\Cref{Def:gpdG}), while $\gpdGfuz$ denotes its fused variant (\Cref{Def:gpdG-fused}).

As already partly evoked, and as will become clear in the course of the proofs, the above diagram of abelian categories is actually the result of taking representation categories on a diagram of bicategories and pseudo-functors, as follows (all notations will be explained in the article):
\[
\xymatrix@C=1.2em@R=1.5em{
&& \Span(\groupoidf) \ar[r]^-{\cat R} \ar[d]
& \biset^{\bifree}\!(\groupoid) {}_{\vcorrect{.8}} \ar@{^(->}[d]
& \ \biset^{\bifree}\!(\group) {}_{\vcorrect{.8}} \ar@{_(->}[l]^-{\;\;\;\sim_+} \ar@{^(->}[d]
\\
\Span(G\sset)
 \ar[r]_-{\simeq}^-{ \underset{\phantom{m}}{\textrm{Prop.\,\ref{Prop:bieq-Gsets-gpdG}} }}
 \ar[d]
  \ar@{}[dr]|{G\ltimes -}
& \Span({\gpdG})
 \ar[r]
 \ar@{}[dr]|-{\mathrm{(forget)}}
  \ar[d]
& \Span(\groupoid;\groupoidf)
  \ar[r]^-{\cat R}
  \ar[d]
& \biset^{\rfree}\!(\groupoid) {}_{\vcorrect{.8}} \ar@{^(->}[d]
& \ \biset^{\rfree}\!(\group) {}_{\vcorrect{.8}} \ar@{_(->}[l]^-{\;\;\;\;\sim_+} \ar@{^(->}[d]
\\
\Span(G\ssetfuz)
\ar[r]^-{\simeq}_-{ \overset{\phantom{m}}{\textrm{Thm.\,\ref{Thm:fusedMackey_vs_gpdGfuz}} }} &
\Span(\gpdGfuz)
 \ar[r]
& \Span(\groupoid) \ar[r]^-{\cat R}_-{ \overset{\phantom{m}}{\textrm{Thm.\,\ref{Thm:Huglo_bifunctor}} }}
& \biset(\groupoid)
& \ \biset(\group) \ar@{_(->}[l]^-{\;\;\;\;\sim_+}
}
\]
More precisely, in order to obtain the first diagram from the second one we must: First, cut down the latter diagram to one of (usual) categories and functors by identifying isomorphic 1-morphisms (\ie by applying the 1-truncation~$\pih$ of \S\,\ref{parag:truncation}). The result is a diagram of pre-additive categories, \ie categories enriched over abelian monoids, and additive functors between them (see \S\,\Cref{parag:k-lin}). Second, we must apply $\Fun_+(-,\Ab)$ throughout, that is we take categories of additive functors into abelian groups. (The arrows marked $\sim_+$ induce equivalences only after this second operation, as they are inclusions in the additive hull; in fact, the resulting equivalences were omitted from the above diagram of abelian categories.)
Note also that the portion of the diagram of bicategories lying to the left of the realization~$\cat R$ results from applying the span-bicategory construction $\Span(-)$ (see \Cref{Rem:span-bicats}) to a suitable diagram of spannable pairs~$(\GG;\JJ)$.

In conclusion, there exists a rich layer of  underlying 2-categorical  information behind these well-known categories of Mackey and biset functors. The exploration and mining of this stratum was begun in \cite{BalmerDellAmbrogio18pp} and deserves to be taken further.

\begin{Rem}
The literature on Mackey and biset functors for finite groups is rather vast and we have not attempted to list all variations on the theme, nor have we tried to assign historical precedence. Some history of the subject can be found in \cite{Webb00} and \cite[\S\,1.4]{Bouc10}. 
\end{Rem}

\begin{Rem}
This article is an offshoot of \cite{BalmerDellAmbrogio18pp}, which developed the basic theory of \emph{Mackey 2-functors}, a categorified version of Mackey functors whose values are additive categories instead of abelian groups. 
We have nonetheless strived to make this article self-contained. 
Indeed, while the use of groupoids and 2-categories arose quite naturally in the categorified context of \cite{BalmerDellAmbrogio18pp}, it is our hope that the present article will show~-- even to readers who do not particularly care for Mackey \underline{2-}functors~-- how the groupoidal viewpoint offers a useful organizing principle for the \emph{usual}, merely abelian-group-valued Mackey functors. Conversely, we also hope that this survey may function as a guide and point of entry into the Mackey literature for those who are already fluent with 2-categories.
\end{Rem}

\begin{Not}
We will work over an arbitrary commutative ring with unit, denoted by~$\kk$. Common choices are the ring of integers, a field, or a nice local ring. 

To be consistent with \cite{BalmerDellAmbrogio18pp}, a span $X\gets S\to Y$ will be always visually understood as going from left to right, \ie as a morphism $X\to Y$; to be consistent with \cite{Bouc10}, a biset ${}_GX_H$ will always be understood as going from right to left, \ie as a morphism $H\to G$
(both are mere conventions). This will make it slightly awkward in \Cref{sec:bisets} where we compare spans and bisets.
\end{Not}

\begin{Ack}
We are grateful to Paul Balmer and Serge Bouc for their interest and for many useful comments on previous versions of this article. We would also like to thank an anonymous referee for their careful reading.
\end{Ack}

\section{Categorical preliminaries}
\label{sec:prelim}%

We collect here some generalities, mostly to fix our terminology (which is standard and consistent with~\cite{BalmerDellAmbrogio18pp}). Categorically confident readers, or those familiar with~\cite{BalmerDellAmbrogio18pp}, should skip ahead to Section~\ref{sec:Mackey-for-2-1-category} and refer back only if necessary.

\begin{parag} \textbf{Groupoids.}
\label{parag:groupoids}
A \emph{groupoid} is a category where all morphisms are invertible. We will identify a group $G$ with the groupoid having a single object~$\bullet$ whose endomorphism monoid is $\End(\bullet)=G$. Under this identification, a homomorphism $f\colon G\to G'$ between groups is the same thing as a functor. 

We will only consider groups and groupoids which are \emph{finite}, that is, which only have finitely many objects and arrows. 

Recall that a groupoid is \emph{connected} if all its objects are isomorphic, in which case the groupoid is equivalent (as a category) to the full subcategory on any one of its objects, which is just a group. Thus every (finite) groupoid~$G$ is equivalent to a (finite) disjoint union of (finite) groups, with the equivalence depending on a chosen set of representative objects for all connected components.
\end{parag}

\begin{parag}\textbf{Bicategories and 2-categories.}
\label{parag:bicats}
(See \eg \cite[\S\,A.1]{BalmerDellAmbrogio18pp}.) A \emph{2-category} $\mathcal C$ is a category enriched in categories, \ie it consists of a collection $\Obj \mathcal C$ of objects (or \emph{0-morphisms}, \emph{0-cells}), together with `Hom' categories $\mathcal C(X,Y)$ for all pairs of objects $X,Y$ (whose objects are the \emph{1-morphisms} or \emph{1-cells} $u\colon X\to Y$ of~$\mathcal C$ and whose arrows are its \emph{2-morphisms} or \emph{2-cells} $\alpha\colon u\Rightarrow v$), and composition functors $\circ= \circ_{X,Y,Z}\colon \mathcal C(Y,Z)\times \mathcal C(X,Y)\to \mathcal C(X,Z)$ subject to the usual (strict) unit and associativity equations. In particular, each object has a left-and-right identity 1-morphism~$\Id_X$.
Our first example is the 2-category 
\[ \gpd \]
consisting of finite groupoids, functors between them, and (necessarily invertible) natural transformations. 

\begin{Not} \label{Not:gamma_x}
Given two parallel homomorphisms $f_1,f_2\colon G\to H$ between groups, considered as functors between one-object groupoids (\S\,\ref{parag:groupoids}), a natural transformation $f_1\Rightarrow f_2$ (a 2-cell in~$\gpd$) is completely determined by its unique component $x\colon \bullet \overset{\sim}{\to}\bullet$, which is an element $x\in H$ such that ${}^xf_1=f_2$ (that is: $xf_1(g)x^{-1}=f_2(g)$ for all $g\in G$). We will write $\gamma_x\colon f_1\Rightarrow f_2$ for this 2-cell.
\end{Not}

A \emph{bicategory}~$\mathcal B$ is a `relaxed' version of a 2-category, where the unit and associativity axioms only hold up to given coherent natural isomorphisms $\Id_X \circ u \cong u \cong u \circ \Id_X$ and $(w \circ v)\circ u\cong w \circ (v \circ u)$. Here `coherent' means that all reasonable diagrams involving these isomorphisms must commute; as a consequence, each bicategory can effectively be replaced by a 2-category which is biequivalent  (see below) to it. 

In any bicategory, the composition of 2-cells within each Hom category is called \emph{vertical} composition, while the effect of applying the composition functors $\circ_{X,Y,Z}$ to 1- or 2-morphisms is called \emph{horizontal} composition. This is reflected by the usual layout in the `cellular' diagram notation for a 2-morphism:  
\[
\xymatrix@R=0pt{
&& \\
X \ar@/^3ex/[rr]^-u \ar@/_3ex/[rr]_-v \ar@{}[rr]|-{\Scell \,\alpha} && Y \\
&&
}
\]
A 2-category is the same thing as a bicategory which is \emph{strict}, that is whose unit and associativity isomorphisms are identity maps. An (ordinary) category, or \emph{1-category}, can be seen as a \emph{discrete} 2-category, that is one whose Hom categories only have identity arrows (hence `are' just sets).
\end{parag}

\begin{Rem} \label{Rem:2-cat-check}
It is often straightforward to directly check that some given data defines a 2-category, \cf \cite[\S\,XII.3]{MacLane98}. For a non-strict \emph{bi}category, on the other hand, some more work may be required.
\end{Rem}

\begin{Exa}[{See e.g.\ \cite[\S7.8]{Borceux94a}}] 
\label{Def:biset_functor_gpd}
Let $G$ and~$H$ be finite groupoids. A \emph{finite biset (or bimodule, distributor, profunctor) of groupoids}, written $U= {}_GU_H \colon H\to G$, is a functor $U \colon H^{\op}\times G\to \set$ to the category of finite sets. There is a bicategory $\biset(\groupoid)$ with finite groupoids as objects, finite bisets $H\to G$ as 1-morphisms from $H$ to~$G$, and natural transformations between them as 2-morphisms.
The horizontal composition of two bisets ${}_GU_H$ and ${}_KV_G$ is given by their tensor product (\ie coend)
\[
V \otimes U
= {}_KV \underset{G}{\otimes}\ U_H
:= \int^{g\in G} V(g,-) \times U(-,g) \colon H^\op\times K \longrightarrow \set\,.
\]
Concretely, the image under $V\otimes U$ of $(h,k)\in \Obj (H^\op\times K)$ is the quotient set
\[
(V\otimes U)(h,k) = \frac{\coprod_{g\in \Obj G} V(g, k) \times U(h,g)}{(\beta \varphi, \alpha)\sim (\beta, \varphi \alpha) \;\; \textrm{ if } \varphi \in G(g_1,g_2) }
\; \in \; \Obj (\set) \,.
\]
Natural transformations between bisets induce maps on these quotients, and this defines the functoriality on maps of horizontal composition.  
\end{Exa}

\begin{Exa} \label{Def:biset_functor_gps}
If the groupoids $G$ and $H$ are just groups, to give a biset ${}_GU_{H}$ as in \Cref{Def:biset_functor_gpd} is the same thing as to give a $G,H$-biset in the sense of \cite{Bouc10}, that is, a set $U$ together with a left action by $G$ and a right action by~$H$ which commute: $(g \cdot u) \cdot h = g \cdot (u\cdot h)$. 
By restricting attention to bisets between groups, we obtain a 1-full and 2-full sub-bicategory of $\biset(\groupoid)$ that we denote by $\biset(\group)$.
\end{Exa}

\begin{parag}\textbf{Internal adjunctions and equivalences.}
\label{parag:internal-adj-equiv}
Two 1-morphisms $\ell\colon X\to Y$ and $r\colon Y\to X$ in a bicategory~$\cat B$ are \emph{adjoint} if there exist 2-morphisms $\eta\colon \Id_X \Rightarrow r\circ \ell$
and $\varepsilon\colon \ell\circ r \Rightarrow \Id_Y$ such that $(\ell \circ \varepsilon)\circ (\eta\circ \ell) = \id_\ell $ and $(\varepsilon \circ r)\circ (r \circ \eta) =\id_r$. An \emph{adjunction}, sometimes written $\ell\dashv r$, is the data of such a quadruple $(\ell, r, \eta, \varepsilon)$. We say the adjunction is an \emph{adjoint equivalence} if $\eta$ and $\varepsilon$ are (necessarily mutually inverse) isomorphisms. More generally, a 1-morphism $\ell\colon X\to Y$ is an \emph{equivalence} if there exist a 1-morphism $r\colon Y\to X$ and two invertible 2-morphisms $r \ell \cong \id_X$ and $\ell r \cong \id_Y$. 
Every equivalence can be completed to an adjoint equivalence.

Inside the 2-category of all categories, functors and natural transformations, these reduce to the usual notions of adjoint functors and equivalence of categories.
\end{parag}

\begin{parag}\textbf{Pseudo-functors and biequivalences.}
\label{parag:pseudofun}
A useful notion of morphism between two bicategories (or even 2-categories) is that of a pseudo-functor. A \emph{pseudo-functor}  $\cat F\colon \cat B \to \cat C$ consists of an assignment $X\mapsto \cat FX$ between the objects of $ \cat B$ and~$\cat C$, functors $\cat F=\cat F_{X,Y}\colon \cat B(X,Y)\to \cat B(\cat FX, \cat FY)$ between their Hom categories, and specified natural isomorphisms $\Id_{\cat FX} \cong \cat F(\Id_X)$ and $\cat F(v) \circ \cat F(u) \cong  \cat F(v \circ u)$ subject to suitable coherence axioms. 
The correct notion of an equivalence between bicategories is that of a \emph{biequivalence}: a pseudo-functor $\cat F\colon \cat B\to \cat C$ such that there exists another pseudo-functor $\cat G\colon \cat C\to \cat B$ and isomorphisms $\cat F\circ \cat G \cong \Id_{\cat C}$ and $\cat G\circ \cat F \cong \Id_{\cat B}$. Here by isomorphism we mean an invertible \emph{modification}, which is the correct notion of a morphism of pseudo-functors (see~\cite[A.1.14]{BalmerDellAmbrogio18pp}). Equivalently, a pseudo-functor $\cat F\colon \cat B\to \cat C$ is a biequivalence iff each functor $\cat F_{X,Y}$ is an equivalence $\cat B(X,Y)\xrightarrow{\sim}\cat C(\cat FX,\cat FY)$  of Hom categories and moreover each object $Y$ of $\cat C$ is equivalent (in the internal sense of \S\,\ref{parag:internal-adj-equiv})  to one of the form~$\cat FX$. 
\end{parag}

\begin{parag}\textbf{The truncation of a bicategory.}
\label{parag:truncation}
The \emph{1-truncation} (also called \emph{classifying category}) of a bicategory~$\cat B$, denoted~$\pih(\cat B)$, is the ordinary category with the same objects as $\cat B$ and whose morphisms are the isomorphism classes of 1-morphisms of~$\cat B$, with the induced composition. That is, we look at 1-morphisms up to invertible 2-morphisms. This operation is functorial, in that it sends pseudo-functors $\cat F$ to ordinary functors~$\pih \cat F$, and preserves composition and identities.
\end{parag}

\begin{Exa} \label{Exa:truncated-bisets}
Recall from \Cref{Def:biset_functor_gps} the bicategory $\biset(\group)$ of finite groups, finite bisets and morphisms of bisets. Its truncation $\pih \biset(\group)$ is precisely Bouc's category of bisets~$\cat{B}$ of~\cite{Bouc10}, and our conventions for composition are consistent with his. The inclusion 2-functor $\biset(\group) \hookrightarrow \biset(\gpd)$ induces a fully faithful functor $\pih \biset(\group) \hookrightarrow \pih \biset(\gpd)$.
\end{Exa}

\begin{parag}\textbf{Finite coproducts in a bicategory.}
\label{parag:coprods-pb}
Let $\cat B$ be any bicategory. An object $\varnothing $ is \emph{initial} if the unique functor
\[
\cat B(\varnothing,T)
 \overset{\sim}{\longrightarrow} 1
\]
to the final category (one object and one identity arrow) is an equivalence for every object $T\in \cat B$; thus, up to isomorphism there is precisely one 1-morphism $\varnothing \to T$. A diagram $X\xrightarrow{i_X} X\sqcup Y \xleftarrow{i_Y} Y$ of 1-morphisms is a \emph{coproduct} if 
\begin{equation} \label{eq:coprod-equiv}
(i_X^*, i_Y^*)\colon \cat B(X\sqcup Y, T) \overset{\sim}{\longrightarrow} \cat B(X,T) \times \cat B(Y,T)
\end{equation}
is an equivalence; in particular, for all 1-morphisms $u\colon X\to T$ and $v\colon Y\to T$ there is up to isomorphism a unique $(u, v)\colon X\sqcup Y\to T$ such that $(u,v)i_X = u$ and $(u,v)i_Y=v$.
Coproducts can be iterated any finite number of times, with similar uniqueness statements, and an initial object can be understood as the empty coproduct. Finite coproducts are unique only up to equivalence, but in many bicategories there are canonical constructions. The coproducts of $\cat B$ yield coproducts in the truncation~$\pih \cat B$.
 
If we require the above two equivalences to be \emph{isomorphisms} of categories, we obtain the more familiar \emph{strict} initial object and coproducts. In all our examples strict (and canonical) versions will be available, but all constructions will work with the above more relaxed notion, which has the advantage of being stable under biequivalence. Mostly we will ignore the difference.
\end{parag}

\begin{parag}\textbf{Iso-comma squares and Mackey squares.}
\label{parag:isocomma}
(See \cite[\S\,2.1-2]{BalmerDellAmbrogio18pp}.) A central role in this article will be played by certain diagrams which will provide a canonical replacement for Mackey formulas, namely \emph{iso-comma squares} and \emph{Mackey squares}. They are, respectively, a strict and a pseudo version of \emph{homotopy pullbacks}. The second has the added advantage of being stable under biequivalence.
 
In a 2-category (or even a bicategory)~$\cat C$, an invertible 2-cell $\gamma$ of the form
\begin{equation} \label{eq:iso-comma-muster}
\vcenter{
\xymatrix@C=14pt@R=14pt{
& (u/v) \ar[dl]_{p_X} \ar[dr]^-{p_Y} & \\
X \ar[dr]_u \ar@{}[rr]|{\underset{\sim}{\oEcell{\gamma}}} && Y \ar[dl]^v \\
&Z &
}}
\end{equation}
is an \emph{iso-comma square} if it is the 2-universal invertible 2-cell sitting on top of $X\xrightarrow{u} Z \xleftarrow{v} Y$. More precisely, for any other invertible 2-morphism $\sigma$ as on the left-hand side
\begin{equation*}
\vcenter{
\xymatrix@C=14pt@R=14pt{
& T \ar[dl]_{r} \ar[dr]^-{s} & \\
X \ar[dr]_u \ar@{}[rr]|{\underset{\sim}{\oEcell{\sigma}}} && Y \ar[dl]^v \\
&Z &
}}
\quad \quad \quad = \quad \quad \quad
\vcenter{
\xymatrix@C=14pt@R=14pt{
&T \ar@/_3ex/[ddl]_{r} \ar@/^3ex/[ddr]^{s} \ar[d]_-{\exists !}^-{t} & \\
& (u/v) \ar[dl]_{p_X} \ar[dr]^-{p_Y} & \\
X \ar[dr]_u \ar@{}[rr]|{\underset{\sim}{\oEcell{\gamma}}} && Y \ar[dl]^v \\
&Z &
}}
\end{equation*}
there is a unique 1-morphism $t$ such that $r= p_Xt$, $s= p_Yt$ and $\gamma t = \sigma$; we write 
\[ t=\langle r,s, \sigma\rangle \] 
and call $r$, $s$ and $\sigma$ the \emph{components} of~$t$. Moreover, it is also required that for any two parallel 1-morphisms $t,t'\colon T\to (u/v)$, the 2-morphisms $\alpha\colon t\Rightarrow t'$ be in bijection with pairs $(\beta_X\colon p_Xt\Rightarrow p_Xt',\beta_Y\colon p_Yt\Rightarrow p_Yt')$ of 2-morphisms between components, the bijection of course being given by $\alpha \mapsto (p_X\alpha, p_Y\alpha)$. 

Similarly, we call a 2-cell as in \eqref{eq:iso-comma-muster} a \emph{Mackey square}\footnote{Beware that in the literature both iso-comma squares and Mackey squares are sometimes called iso-comma squares, or pseudo-pullbacks, or even just pullbacks.} if it satisfies almost the same 2-universal property as above, with the following difference: for each triple $\langle r,s, \sigma\rangle$, there exist a $t\colon T\to (u/v)$ and two 2-isomorphisms $\varphi\colon r\Rightarrow p_Xt$ and $\psi\colon p_Yt\Rightarrow s$ such that $\sigma=(u\varphi)(\gamma t)(v \psi)$. 
(It follows that such a $t$ is unique up to a non-unique isomorphism.) 

\begin{Exa} \label{Exa:iso-comma-cats}
In the 2-category $\Cat$ of small categories, functors and natural transformations, iso-comma squares have the following canonical construction.
The category $(u/v)$ has for objects all triples $(x,y,\gamma)$ with $x\in \Obj X$, $y\in \Obj Y$ and $\gamma \colon u(x)\overset{\sim}{\to}v(y)$ an isomorphism in~$Z$; a  morphism $(x,y,\gamma)\to (x',y', \gamma')$ is a pair $(\alpha , \beta)$ with $\alpha \in X(x,x')$, $\beta \in Y(y,y')$ and such that $v(\beta) \circ \gamma = \gamma' \circ u(\alpha)$ in~$Z$. The two functors $p_X$ and $p_Y$ are the evident projections $(x,y,\gamma)\mapsto x$ and $(x,y,\gamma)\mapsto y$, and the natural isomorphism $\gamma\colon up_X\Rightarrow vp_Y$ has the `tautological' component $\gamma$ at the object $(x,y,\gamma)$.
If $X,Y,Z$ happen to all be finite groupoids then so is $(u/v)$, hence this construction also provides iso-comma squares for the 2-category~$\gpd$.
\end{Exa}

\begin{Rem}[{See \cite[2.1.11-13]{BalmerDellAmbrogio18pp}}]  \label{Rem:iso-comma-Mack-characterization}
Iso-comma squares and Mackey squares in any 2-category $\cat C$ can be nicely characterized in terms of iso-comma squares of their Hom categories, built in $\Cat$ as in \Cref{Exa:iso-comma-cats}. Namely, consider a 2-cell in~$\cat C$:
\begin{equation} \label{eq:candidate-iso-comma}
\vcenter{
\xymatrix@C=14pt@R=14pt{
& P \ar[dl]_{p} \ar[dr]^-{q} & \\
X \ar[dr]_u \ar@{}[rr]|{\underset{\sim}{\oEcell{\gamma}}} && Y \ar[dl]^v \\
&Z &
}}
\end{equation}
For every object $T\in \Obj \cat C$, we can apply $\cat C(T,-)$ to it in order to obtain the following comparison functor:
\begin{equation} \label{eq:comparison-iso-commas}
\cat C(T, P) \longrightarrow \big(\cat C(T,u) / \cat C(T,v)\big) , \quad \quad  t \longmapsto (pt, qt, \gamma t)
\end{equation}
(this is the unique functor with components $\langle \cat C(T,p) , \cat C(T,q) , \cat C(T,\gamma) \rangle$).
Then, as we see directly from the definitions, \eqref{eq:candidate-iso-comma} is an iso-comma square (resp.\ a Mackey square) iff \eqref{eq:comparison-iso-commas} is an isomorphism of categories (resp.\ an equivalence).
\end{Rem}

\begin{Rem}
Iso-comma squares are, in particular, Mackey squares.
If the iso-comma square over $X \xrightarrow{u} Z \xleftarrow{v}Y$ exists, then a square \eqref{eq:candidate-iso-comma} is a Mackey square iff the comparison functor $\langle p,q,\gamma \rangle $ into the iso-comma square is an equivalence.
\end{Rem}

\end{parag}

\begin{parag}\textbf{Comma 2-categories.}
\label{parag:comma2cat}
Given an object $C$ in a 2-category~$\cat C$, the \emph{comma 2-category} $\cat C_{/C}$ is the 2-category defined as follows (this also works for general bicategories, but we will not need it). Its objects are pairs $(X,p)$ of an object $X$ and a specified 1-cell $p\colon X\to C$ of~$\cat C$. A 1-cell $(X,p)\to (Y,q)$ is a pair $(u,\theta)$ where $u\colon X\to Y$ is a 1-cell and $\theta\colon qu\overset{\sim}{\Rightarrow} p$ an invertible 2-cell of~$\cat C$ (as $\theta$ is invertible, the choice of its direction is a matter of convention). Finally, a 2-cell $(u,\theta)\Rightarrow (u',\theta')$ is a 2-cell $\alpha\colon u\Rightarrow u'$ such that $\theta' (q\alpha) = \theta$ in~$\cat C$:
\[
\vcenter{
\xymatrix@R=8pt{
&& \\
X 
 \ar@/^2ex/[rr]^-u 
  \ar@/_2ex/[rr]_-{u'} 
   \ar@{}[rr]|-{\Scell \,\alpha}
    \ar[dddr]_p & \ar@{}[ddd]|>>>>>>{\SWcell \; \theta'} &
 Y \ar[dddl]^q \\
&& \\
&& \\
&C &
}}
\quad = \quad
\vcenter{
\xymatrix@R=8pt{
& \ar@{}[dddd]|{\SWcell \; \theta} & \\
X 
 \ar@/^2ex/[rr]^-u 
   \ar[dddr]_p &&
 Y \ar[dddl]^q \\
&& \\
&& \\
&C &
}}
\]
The vertical and horizontal compositions in $\cat C_{/C}$ are inherited from those of $\cat C$ in the evident way.
There is a forgetful 2-functor $\cat C_{/C}\to \cat C$ which simply forgets the $p$ and $\theta$ parts.

\begin{Rem} \label{Rem:coprods-comma2cats}
One verifies easily that $\cat C$ admits coproducts (\S\,\ref{parag:coprods-pb}) if and only if $\cat C_{/C}$ does so for every $C\in \cat C$. In this case, the forgetful 2-functor $\cat C_{/C}\to \cat C$ preserves and reflects them, in the sense that a diagram
\[
\xymatrix{
X_1 \ar[r] \ar[dr] &
 X
  \ar@{}[dl]|(.34){{}^{\simeq}\,\SWcell}
   \ar@{}[dr]|(.34){\SEcell\,{}^{\simeq}}
    \ar[d] &
X_2 \ar[l] \ar[dl] \\
& C&
}
\]
is a coproduct in $\cat C_{/C}$ iff the top row is a coproduct in~$\cat C$.
\end{Rem}

\begin{Rem} \label{Rem:pseudofun-for-extensivity}
If $\cat C$ admits coproducts $X\sqcup Y$ of any two objects, then any choice of adjoint quasi-inverses for all the functors \eqref{eq:coprod-equiv} (or no choice at all if the coproducts are strict) defines a pseudo-functor
\[
\cat C_{/X} \times \cat C_{/Y} \longrightarrow \cat C _{/ X \sqcup Y} \,, \quad (A\overset{a}{\to} X, B \overset{b}{\to} Y) \mapsto (A\sqcup B \xrightarrow{a\sqcup b}X\sqcup Y)
\]
in an evident way. Any two choices, of course, yield isomorphic pseudo-functors in a suitable sense.
\end{Rem}
\end{parag}

\begin{parag}\textbf{Additive, semi-additive and $\kk$-linear categories.}
\label{parag:k-lin}
If $\kk$ is a commutative ring with unit, a \emph{$\kk$-linear category} $\cat C$ is a category enriched in $\kk$-modules; this means that each Hom set $\cat C(X,Y)$ carries the structure of a $\kk$-module and the composition maps are all $\kk$-bilinear.

A $\kk$-linear category is \emph{additive} if moreover it has direct sums  (a.k.a.\ biproducts) $X_1\oplus \ldots \oplus X_n$ for every finite set of its objects, including an empty direct sum~$0$, a.k.a.\ a zero object. Direct sums are both categorical product and coproduct diagrams, and morphisms between them are determined by their matrix of components, and can be composed according to matrix multiplication. 
Direct sums can be used to recover the underlying additive monoid structure of each Hom set $\cat C(X,Y)$, since its zero element must be the unique map $x\to 0\to y$ and the sum of $f,g\in \cat C(X,Y)$ must be the composite
\[
f+ g = \Big( \xymatrix{ X \ar[r]^-{({}^1_1)} & X \oplus X \ar[r]^-{\big( {}^f_0 \; {}^0_g\big)} & Y \oplus Y \ar[r]^-{(1\;1)} & Y  } \Big)
\]
(see \cite[VIII.2]{MacLane98} or \cite[A.5]{BalmerDellAmbrogio18pp}).

We call \emph{semi-additive} a category with finite direct sums which is enriched in abelian monoids (in the unique way possible, as above).

A functor $F\colon \cat C\to \cat D$ is \emph{$\kk$-linear} if each component $F_{X,Y}\colon \cat C(X,Y)\to \cat D(FX,FY)$ is a $\kk$-linear map. 
Similarly, if $\cat C$ and $\cat D$ are (only) enriched in abelian monoids, a functor $F\colon \cat C\to \cat D$ is said to be \emph{additive} if it preserves the addition and zero element of each Hom monoid.
Note that additive functors (hence in particular $\kk$-linear ones) always preserve direct sum diagrams when they exist. 

We will denote by $\Fun_\kk(\cat C,\cat D)$ the category of $\kk$-linear functors $\cat C\to \cat D$ and natural transformations. Similarly, if $\cat C,\cat D$ are (only) enriched in abelian monoids, we will denote by $\Funplus(\cat C, \cat D)$ the category of additive functors between them.
\end{parag}

\begin{Cons}[{The $\kk$-linearization}] \label{Cons:kk-linearization}
Let $\cat C$ be a category enriched in additive monoids. Its \emph{$\kk$-linearization} is the $\kk$-linear category 
\[ \kk\cat C \]
 defined as follows. Its objects are the same, $\Obj (\kk\cat C):= \Obj(\cat C)$. Its Hom $\kk$-modules $\kk\cat C(X,Y)$ are obtained by first building the Grothendieck group completion $\cat C(X,Y)^{\pm}$ (the abelian group of formal differences) and then extending scalars: $\kk\cat C(X,Y):= \kk \otimes_\mathbb Z \cat C(X,Y)^{\pm}$. The composition maps of $\kk \cat C$ are the unique $\kk$-bilinear maps extending the composition maps of~$\cat C$ along the canonical maps $\cat C(X,Y)\to \kk\cat C(X,Y)$, $f\mapsto 1\otimes (f-0)$.

The latter canonical maps also define an evident functor $\cat C\to \kk\cat C$ having the universal property that it induces, by precomposition, an isomorphism of categories
\[
\Fun_\kk (\kk\cat C, \cat D) \overset{\sim}{\longrightarrow}  \Funplus (\cat C, \cat D)
\]
for any given $\kk$-linear category~$\cat D$. 

Clearly, if $\cat C$ has direct sums (\ie is semi-additive) then these remain direct sums in the $\kk$-linearization $\kk\cat C$, so that the latter is an \emph{additive} $\kk$-linear category.
\end{Cons}

\begin{Rem}
Note that the above construciton is not the same as the \emph{free} $\kk$-linearization, which can be performed on any category by taking the free $\kk$-module on each Hom set.  It is in fact important for us that we remember the existing (semi-) additive structure already present on our categories of spans and bisets.
\end{Rem}

\section{Mackey functors for (2,1)-categories}
\label{sec:Mackey-for-2-1-category}%

As announced in the introduction, we want to define a Mackey functor to be a linear functor defined on a suitable category of spans of groupoids. Moreover, we wish to allow variations with sub- and comma-2-categories. A natural setting for constructing categories of spans, and covering all such cases, is that of an \emph{extensive (2,1)-category} $\GG$ with enough Mackey squares and equipped with a distinguished sub-2-category $\JJ$ with suitable closure properties (what we call a \emph{spannable pair}; see \Cref{Def:spannable}).
As the next few pages may appear a little abstract, we urge the reader to keep in mind the example $\GG=\JJ=\gpd$ of all finite groupoids, where Mackey squares are provided by the concrete iso-comma construction of \Cref{Exa:iso-comma-cats}, which allow us to check all the following claims by direct computations. This particular example will also be revisited in the next section in much detail. But for now, let us bask in some glorious generality:

\begin{Def} \label{Def:21cat}
A \emph{$(2,1)$-category} is\footnote{Roughly speaking, the general pattern in higher category theory is that an \emph{$(n,k)$-category} has morphisms between morphisms between morphisms etc.\ up to level~$n$, but only those up to level~$k$ are allowed to be non-invertible. Moreover, \emph{$n$-category} is short for $(n,n)$-category (with the notable exception that Lurie \cite{Lurie09} calls  \emph{$\infty$-categories} his models of $(\infty,1)$-categories).} 
a strict 2-category, as in \S\,\ref{parag:bicats}, where moreover all 2-morphisms are invertible. 
\end{Def}

\begin{Def} [{\cite{BungeLack03}}] \label{Def:extensive} 
A 2-category $\cat E$ is \emph{extensive} if it admits all finite coproducts (see \S\,\ref{parag:k-lin}) and if moreover, for any pair of objects $X,Y$ the pseudo-functor 
\begin{equation} \label{eq:comparison-pseudofun-commas}
\cat E_{/X} \times \cat E_{/Y} \overset{\sim}{\longrightarrow} \cat E _{/ X \sqcup Y} \,, \quad (A\overset{a}{\to} X, B \overset{b}{\to} Y) \mapsto (A\sqcup B \xrightarrow{a\sqcup b}X\sqcup Y)
\end{equation}
induced on comma 2-categories (see \S\,\ref{parag:comma2cat} and \Cref{Rem:pseudofun-for-extensivity}) by taking coproducts is a biequivalence (see \S\,\ref{parag:pseudofun}). 
By \cite[Thm.\,2.3]{BungeLack03}, a 2-category with finite coproducts is extensive if and only if: 1) it admits Mackey squares along all coproduct inclusions, and 2) it has the property that in any diagram of the form
\begin{equation} \label{eq:PB-vs-coprod}
\vcenter{
\xymatrix{
T_A \ar[d]_{a} \ar[r] & T \ar[d]^c \ar@{}[dl]|{\simeq} \ar@{}[dr]|{\simeq} & T_B \ar[l] \ar[d]^{b} \\
A \ar[r]^-{i_A} & A\sqcup B & B \ar[l]_-{i_B} 
}}
\end{equation}
the top row is a coproduct (\ie $T_A \sqcup T_B \to T$ is an equivalence) iff the two squares are Mackey squares.
(Beware that in \cite{BungeLack03} Mackey squares are simply called `pullbacks' and the invertible 2-cells are omitted from all such diagrams.)
\end{Def}

\begin{Rem} \label{Rem:extensive-cats}
The notion of extensive category was introduced in \cite{CarboniLackWalters93} to capture, at the categorical level, the intuition of coproducts behaving like set-theoretical disjoint unions. A 1-category is \emph{extensive} if it admits all finite coproducts and if the 1-categorical analogues of \eqref{eq:comparison-pseudofun-commas} are equivalences (or equivalently: if it is extensive when seen as a locally discrete 2-category). Basic examples include elementary toposes such as $\set$ or $G\sset$.
This definition was extended to 2- (and bi-) categories in~\cite{BungeLack03}, with small categories $\Cat$ and (finite) groupoids~$\gpd$ providing basic examples. The definition pins down a certain compatibility between coproducts and (weak) pullbacks, which is what we need here in order to form nice categories of spans (\cf also \cite{PanchadcharamStreet07}). 
\end{Rem}

More precisely, we need the following two lemmas:

\begin{Lem} \label{Lem:Mackey-vs-PB-1}
Let $\cat E$ be an extensive 2-category (\Cref{Def:extensive}). Then squares of the form
\begin{equation} \label{eq:MackeyPB-kind1}
\vcenter{
\xymatrix@!C@C=0pt@R=14pt{
& A \ar[dl]_{u} \ar[dr]^-{i_A} & \\
X \ar[dr]_{i_X} \ar@{}[rr]|{\simeq} && A\sqcup B \ar[dl]^{u \sqcup v} \\
&X\sqcup Y &
}}
\quad\quad \textrm{ and } \quad\quad
\vcenter{
\xymatrix@!C@C=0pt@R=14pt{
& B \ar[dl]_{i_B} \ar[dr]^-{v} & \\
A\sqcup B \ar[dr]_{u\sqcup v} \ar@{}[rr]|{\simeq} && Y \ar[dl]^{i_Y} \\
&X\sqcup Y &
}}
\end{equation}
are Mackey squares for all 1-morphisms $u,v$.
\end{Lem}

\begin{proof}
This is immediate from the characterization of extensive 2-categories recalled in \Cref{Def:extensive}.
\end{proof}

\begin{Lem} \label{Lem:Mackey-vs-PB-2}
In an extensive 2-category~$\cat E$, if the two squares on the left are Mackey squares
\begin{equation} \label{eq:MackeyPB-kind2}
\vcenter{
\xymatrix@C=14pt@R=14pt{
& P_\ell \ar[dl]_{p_\ell} \ar[dr]^-{q_\ell} & \\
X_\ell \ar[dr]_{u_\ell} \ar@{}[rr]|{\underset{\sim}{\oEcell{\gamma_\ell}}} && Y \ar[dl]^{w} \\
&Z &
}}
\;\;\;\; (\ell\in \{1,2\})
\quad\quad \rightsquigarrow \quad\quad
\vcenter{
\xymatrix@!C@C=-3pt@R=14pt{
& P_1\sqcup P_2 \ar[dl]_{p_1 \sqcup p_2} \ar[dr]^-{(q_1,q_2)} & \\
X_1\sqcup X_2 \ar[dr]_{(u_1,u_2)} \ar@{}[rr]|{\underset{\sim}{\oEcell{(\gamma_1,\gamma_2)}}} && Y \ar[dl]^{w} \\
&Z &
}}
\end{equation}
then so is the induced square on the right.
\end{Lem}

\begin{proof}
Consider two Mackey squares as in \eqref{eq:MackeyPB-kind2}. By \Cref{Rem:iso-comma-Mack-characterization}, we must show that the functor into the iso-comma category
\begin{align} \label{eq:comparison-E2}
\cat E(T, P_1\sqcup P_2) &\longrightarrow \big(\cat E(T,(u_1,u_2)) / \cat E(T,w)\big) \\ 
 t &\longmapsto (\,(p_1\sqcup p_2)t, (q_1,q_2)t, (\gamma_1,\gamma_2)t \, ) \nonumber
\end{align}
is an equivalence for every~$T$. 
Let $(x,y,\varphi)$ be any object in the target, that is $x\colon T\to X_1\sqcup X_2$ and $y\colon T\to Y$ are 1-cells and $\varphi\colon (u_1,u_2)x\Rightarrow wy$ is a 2-cell in~$\cat E$. By the extensivity of~$\cat E$, there exists a decomposition $T_1\sqcup T_2\overset{\sim}{\to} T$  which identifies $x$ with the coproduct of some $x_1\colon T_1\to X_1$ and $x_2\colon T_2\to X_2$. By precomposing with the canonical inclusions $T_\ell\to T$ ($\ell=1,2$), we can also write $y$ and $\varphi$ in their two components $y_\ell\colon T_\ell\to Y$ and
\[ 
\vcenter{
\xymatrix@C=14pt@R=14pt{
& T_\ell \ar[dl]_{x_\ell} \ar[dr]^-{y_\ell} & \\
X_\ell \ar[dr]_{u_\ell} \ar@{}[rr]|{\underset{\sim}{\oEcell{\varphi_\ell}}} && Y \ar[dl]^w \\
&Z &
}}
\]
respectively. Using the two given Mackey squares on $\overset{u_\ell\;\,}{\to} \overset{w}{\gets}$, the above 2-cells $\varphi_\ell$ can be written as pastings of the form
\[
\vcenter{
\xymatrix@C=14pt@R=14pt{
& T_\ell \ar[d]^{t_\ell} \ar@/_3ex/[ddl]_{x_\ell} \ar@/^3ex/[ddr]^{y_\ell} \ar@{}[ddl]|(.43){\simeq} \ar@{}[ddr]|(.43){\simeq} & \\
& P_\ell \ar[dl]_{p_\ell} \ar[dr]^-{q_\ell} & \\
X_\ell \ar[dr]_{u_\ell} \ar@{}[rr]|{\underset{\sim}{\oEcell{\gamma_\ell}}} && Y \ar[dl]^{w} \\
&Z &
}}
\]
for some $t_\ell\colon T_\ell\to P_\ell$ (for $\ell= 1,2$). Let $t:= (T \simeq T_1\sqcup T_2 \xrightarrow{t_1\sqcup t_2} P_1\sqcup P_2)$. By construction, the image of $t$ under the functor \eqref{eq:comparison-E2} is isomorphic to the given triple $(x,y,\varphi)$. 
This shows that \eqref{eq:comparison-E2} is essentially surjective.

A similar extensivity argument for 2-cells $\alpha\colon t\Rightarrow t'$ shows fully faithfulness, whence the desired equivalence. 
The remaining details are straightforward and are left to the reader.
\end{proof}

\begin{Def} [{Spannable pair}] 
\label{Def:spannable}
We call  \emph{spannable pair} a pair $(\GG;\JJ)$ where:
\begin{enumerate}[(1)]
\item  $\GG$ is a $(2,1)$-category (\Cref{Def:21cat}) which we assume essentially small, \ie the equivalence classes of objects form a set and every Hom category is small; 
\item $\GG$ is also an extensive 2-category (\Cref{Def:extensive}); and
\item $\JJ$ is a distinguished class of 1-cells of~$\GG$;
\end{enumerate}
and the pair $(\GG;\JJ)$ satisfies the following three axioms:
\begin{enumerate}[\rm(a)] 
\item \label{it:J-closures}
The class $\JJ$ contains the equivalences of $\GG$ (see \S\,\ref{parag:pseudofun}) and is closed under horizontal composition and under taking isomorphic 1-cells. 
In particular, we may identify $\JJ$ with the corresponding \emph{2-full} sub-2-category of~$\GG$ (\ie if $\alpha\colon u\Rightarrow v$ is a 2-cell of $\GG$ with $u,v$ in $\JJ$ then $\alpha$ is also in the sub-2-category~$\JJ$).
\item \label{it:Mackey-sq}
For any 1-cells $X \xrightarrow{i}  Z \xleftarrow{u} Y$ with $i\in \JJ$, there exists in $\GG$ a Mackey square
\[ 
\vcenter{
\xymatrix@C=14pt@R=14pt{
& P \ar[dl]_{p} \ar[dr]^-{q} & \\
X \ar[dr]_i \ar@{}[rr]|{\underset{\sim}{\oEcell{\gamma}}} && Y \ar[dl]^u \\
&Z &
}}
\]
(see \S\,\ref{parag:isocomma}), and moreover $q \in \JJ$.
\item \label{it:coprods}
For any finite set  $\{u_\ell\colon X_\ell\to Y\}_\ell$ of 1-cells of~$\GG$ with common target, the 1-cell $(u_\ell)_\ell\colon \coprod_\ell X_\ell \to Y$ is in~$\JJ$ iff each $u_\ell$ is.
\end{enumerate}
 In the special case where $\GG=\JJ$ we write $\GG:= (\GG;\JJ)$ for short and call $\GG$ a \emph{spannable (2,1)-category}. 
\end{Def}

\begin{Rem} \label{Rem:special-JJ}
By \eqref{it:J-closures} and~\eqref{it:coprods}, the canonical 1-cells $X \xrightarrow{i_X} X\sqcup Y  \xleftarrow{i_Y} Y$ of a coproduct belong to~$\JJ$, because the identity $\Id_{X\sqcup Y}= (i_X,i_Y)$ does.
Similarly, for every $X$ the unique 1-cell $\varnothing\to X$ from the initial object ($=$~empty coproduct) is in~$\JJ$, because it factors as  the composite $\varnothing \smash{\xrightarrow{i_\varnothing}} \varnothing \sqcup X \simeq X$ of two 1-cells in~$\JJ$.
\end{Rem}

\begin{Rem} \label{Rem:theG=Jcase}
Notice that a spannable 2-category $\GG$ (\ie the case of a spannable pair with $\JJ=\GG$, which will cover most of our explicit examples) is precisely the same thing as an essentially small extensive (2,1)-category admitting arbitrary Mackey squares. Indeed, the closure properties \eqref{it:J-closures}--\eqref{it:coprods} are then automatically satisfied.
\end{Rem}

\begin{Cons}[{The category of spans $\spancat(\GG;\JJ)$}] 
\label{Cons:1-spans}
Let $(\GG;\JJ)$ be a spannable pair as in \Cref{Def:spannable}. We construct a category
\[\spancat (\GG;\JJ) \]
whose objects are the same as those of $\GG$, and where a morphism $X\to Y$ is the equivalence class of a span of 1-morphisms of~$\GG$
\[
\xymatrix@R=.5em{
&S \ar[dl]_u \ar[dr]^{i \,\in\, \JJ} & \\
X \ar@{..>}[rr] & & Y
}
\]
where the `forward' one belongs to~$\JJ$. Two spans $X \xleftarrow{u} S \xrightarrow{i} Y$ and $X \xleftarrow{u'} S' \xrightarrow{i'} Y$ are \emph{equivalent} iff there exist in $\GG$ a diagram
\[
\vcenter{
\xymatrix@R=.5em@C=4em{ 
& S \ar[dl]_u \ar[dr]^-i \ar[dd]^s_{\simeq} & \\
X \ar@{}[r]|{\;\;\;\;\alpha\SEcell}  && Y \ar@{}[l]|{\SWcell\beta\;\;\;\;} \\
& S' \ar[ul]^{u'} \ar[ur]_{i'} &
}}
\]
where $s\colon S\overset{\sim}{\to} S'$ is an equivalence and $\alpha\colon u\Rightarrow u's$ and $\beta\colon i \Rightarrow i's$ are 2-morphisms (which are invertible, this being a (2,1)-category). The composition of two (equivalence classes of) spans $[X \xleftarrow{u} S \xrightarrow{i} Y]$ and $[Y \xleftarrow{v} T \xrightarrow{j} Z]$ is the span $[ X \xleftarrow{up} P \xrightarrow{jq} Z]$ which is obtained by constructing a Mackey square in the middle:
\[ 
\vcenter{
\xymatrix@R=1em@C=3em{
&& P \ar[dl]_{p} \ar[dr]^-{q} 
   && \\
& S \ar[dl]_u \ar[dr]_i \ar@{}[rr]|{\underset{\sim}{\oEcell{\gamma}}} && T \ar[dl]^v \ar[dr]^j & \\
X \ar@{..>}[rr] &&Y \ar@{..>}[rr] && Z
}}
\]
(Note that such a Mackey square exists and $jq\in \JJ$ by the hypotheses \eqref{it:J-closures} and~\eqref{it:Mackey-sq} on a spannable pair.) The identity morphism of an object $X$ is given by the span $[X\xleftarrow{\Id} X \xrightarrow{\Id} X]$.
\end{Cons}

\begin{Prop} \label{Prop:1-spans}
\Cref{Cons:1-spans} yields a well-defined category $\spancat(\GG;\JJ)$, which moreover is semi-additive (see \S\,\ref{parag:k-lin}). Explicitly, the zero object is given by the initial object $\varnothing$ (empty coproduct) of~$\GG$ and the direct sum diagram for two objects $X,Y$ is given by the four spans (\ie two spans considered in both directions)
\[
\xymatrix@R=1em@C=3em{
& X \ar@{=}[dl] \ar[dr]^{i_X} && Y \ar[dl]_{i_Y} \ar@{=}[dr] & \\
X \ar@{..>}@<.4ex>[rr] \ar@{<..}@<-.4ex>[rr] &&  X\sqcup Y && Y \ar@{..>}@<-.4ex>[ll] \ar@{<..}@<.4ex>[ll] \\
&  X \ar@{=}[ul] \ar[ur]_{i_X} && Y \ar[ul]^{i_Y} \ar@{=}[ur] &
}
\]
where $X \xrightarrow{i_X} X\sqcup Y \xleftarrow{i_Y} Y $ is the coproduct of $X$ and $Y$ in~$\GG$. The spans 
\[
\vcenter{
\xymatrix@R=.5em{
& \varnothing \ar[dl] \ar[dr] & \\
X \ar@{..>}[rr] & & Y
}}
\quad\quad \textrm{ and } \quad\quad
\vcenter{
\xymatrix@R=.5em{
&S_1\sqcup S_2 \ar[dl]_{(u_1,u_2)\;\;} \ar[dr]^{\;\;(i_1 , i_2)} & \\
X \ar@{..>}[rr] & & Y
}}
\]
provide the zero map $X\to Y$ and the sum of two maps $[X\xleftarrow{u_\ell} S_\ell \xrightarrow{i_\ell} Y]$, $\ell\in \{1,2\}$.
\end{Prop}

\begin{proof}
These are rather straightforward verifications, as follows. 

To see that the composition of spans is well-defined, associative and unital, one must repeatedly employ, in a routine way, the 2-universal property of Mackey squares and the closure properties~\eqref{it:J-closures} and~\eqref{it:Mackey-sq} of the spannable pair~$(\GG;\JJ)$. 

Thanks to \Cref{Rem:special-JJ} and property~\eqref{it:coprods}, we may form the zero span and the sum of two spans as indicated. The universal property of coproducts ensures that this operation is associative and unital on each Hom set. To verify that the composition of $\spancat(\GG;\JJ)$ preserves zero spans, it suffices to notice that the squares of the form
\[ 
\vcenter{
\xymatrix@C=18pt@R=14pt{
& \varnothing \ar[dl] \ar[dr] & \\
\varnothing \ar[dr] \ar@{}[rr]|{\simeq} && B \ar[dl]^{v} \\
&Y &
}}
\quad\quad \textrm{ and }\quad \quad
\vcenter{
\xymatrix@C=18pt@R=14pt{
& \varnothing \ar[dl] \ar[dr] & \\
A \ar[dr]_{u} \ar@{}[rr]|{\simeq} && \varnothing \ar[dl] \\
&Y &
}}
\]
are Mackey squares by \Cref{Lem:Mackey-vs-PB-1} (set $A=X=\varnothing$ in the first one and $B=Y=\varnothing$ in the second one).
Similarly, composition preserves sums of spans because the sum of two Mackey squares, as in \eqref{eq:MackeyPB-kind2} (and also its left-right mirror version), is again a Mackey square by \Cref{Lem:Mackey-vs-PB-2}. Thus $\spancat(\GG;\JJ)$ is enriched in abelian monoids.

Finally, to verify the four equations for the claimed biproducts we can use that 
\[ 
\vcenter{
\xymatrix@C=14pt@R=14pt{
& X \ar[dl]_{\Id} \ar[dr]^-{\Id} & \\
X \ar[dr]_{i_X} \ar@{}[rr]|{=} && X \ar[dl]^{i_X} \\
&X\sqcup Y &
}}
\quad\quad \textrm{ and }\quad \quad
\vcenter{
\xymatrix@C=14pt@R=14pt{
& \varnothing \ar[dl] \ar[dr] & \\
X \ar[dr]_{i_X} \ar@{}[rr]|{\simeq} && Y \ar[dl]^{i_Y} \\
&X\sqcup Y &
}}
\]
are Mackey squares, which again follows from \Cref{Lem:Mackey-vs-PB-1}, this time by specializing the two squares in \eqref{eq:MackeyPB-kind1} to $u=\Id_X$ and $B=\varnothing$. Note that the four spans comprising a biproduct diagram are all permissible by \Cref{Rem:special-JJ}.
\end{proof}

\begin{Rem} \label{Rem:span-bicats}
The category $\spancat(\GG;\JJ)$ of \Cref{Cons:1-spans} is only the shadow of a richer structure. Indeed, it is precisely the 1-truncation (see \S\,\ref{parag:truncation})
\[
\spancat(\GG;\JJ) = \pih (\Span(\GG;\JJ))
\]
 of a \emph{bi}category of spans $\Span(\GG;\JJ)$. This is explained in \cite[Ch.\,4]{BalmerDellAmbrogio18pp} in all details, for the situation where the necessary Mackey squares of $\GG$ are actually iso-comma squares; the additive aspects are covered in \cite[Ch.\,7]{BalmerDellAmbrogio18pp} but only for the example  $\GG=\gpd$ of groupoids.
It is straightforward to generalize the construction of $\Span(\GG;\JJ)$ to any spannable pair $(\GG;\JJ)$ as in \Cref{Def:spannable}, but it involves some choices (namely, in order to obtain specified composition functors one must choose a pseudo-inverse for each of the equivalences~\eqref{eq:comparison-iso-commas}).
\end{Rem}

\begin{Rem} \label{Rem:alt_descr_Span}
There is an alternative way of describing the category $\spancat(\GG;\JJ)$, where the operations of forming spans and 1-truncating are permuted.
In this other picture, we begin by forming the truncated (ordinary) category $\pih \GG$, where a morphism is an isomorphism class $[u]$ of 1-morphisms in~$\GG$. Then we consider spans in $\pih \GG$, that is diagrams 
\[ 
X \overset{[u]}{\longleftarrow}S \overset{[i]}{\longrightarrow} Y
\]
with $i\in \JJ$, and we identify isomorphic spans, where two spans $([u],[i])$ and $([u'],[i'])$ are \emph{isomorphic} if there exists a commutative diagram
\begin{align*}
\xymatrix@C=24pt@R=10pt{
& S \ar[ld]_-{[u]} \ar[dr]^-{[i]} \ar[dd]^-{[f]}_\simeq & \\
X && Y \\
& S' \ar[ul]^-{[u']} \ar[ru]_-{[i']} &
}
\end{align*}
in $\pih{\GG}$ with $[f]$ invertible. Note however that, in general, one cannot define the composition of $\spancat(\GG;\JJ)$ purely in terms of the ordinary category~$\pih \GG$, because one would still need to remember the (images in $\pih \GG$ of the) Mackey squares of~$\GG$. (Sometimes said images admit an intrinsic characterization  in~$\pih \GG$, but not always.)
\end{Rem}

\begin{Def}[{The $\kk$-linear category of spans, $\spancat_\kk(\GG;\JJ)$}]
\label{Cons:1-spansk}

We define the \emph{$\kk$-linear category of spans in $\GG$ with respect to~$\JJ$} to be the $\kk$-linearization (\Cref{Cons:kk-linearization}) of the semi-additive category $\spancat(\GG;\JJ)$ of \Cref{Prop:1-spans}:
\[ \spancat_\kk(\GG;\JJ) := \kk \big( \spancat (\GG;\JJ) \big) \,. \] 
Thus the objects remain the same and the morphisms of $\spancat_\kk(\GG;\JJ)$ are formal $\kk$-linear combinations of morphisms of $\spancat(\GG;\JJ)$, in a way that preserves sums. 
\end{Def}

The span construction is functorial in the following evident sense:

\begin{Prop} \label{Prop:functoriality-of-spancat}
Let $\cat F\colon (\GG;\JJ)\to (\GG';\JJ')$ be a \emph{morphism of spannable pairs}, by which we mean a pseudo-functor (\eg a 2-functor) $\cat F\colon \GG\to \GG'$ preserving (up to equivalence) finite coproducts and Mackey squares and such that $\cat F(\JJ)\subseteq \JJ'$. Then $\cat F$ induces a $\kk$-linear functor
on the span categories of \Cref{Cons:1-spansk}
\[
\spancat_\kk(\cat F) \colon \spancat_\kk(\GG;\JJ) \longrightarrow \spancat_\kk(\GG';\JJ')
\]
sending $[X \xleftarrow{u} S \xrightarrow{i} Y]$ to $[\cat F X \xleftarrow{\cat Fu} \cat FS \xrightarrow{\cat Fi} \cat FY]$.
Moreover 
\[ \spancat_\kk(\cat F_2 \circ \cat F_1)=\spancat_\kk(\cat F_2) \circ \spancat_\kk( \cat F_1)
\quad \textrm{ and } \quad
\spancat_\kk(\Id_{(\GG;\JJ)})=\Id_{\,\spancat_\kk(\GG;\JJ)}
\]
and if $\cat F$ is a biequivalence then $\spancat_\kk(\cat F)$ is an equivalence of $\kk$-linear categories.
\end{Prop}

\begin{proof}
All claims are immediate from the constructions.
\end{proof}

Under a Krull-Schmidt-type finiteness hypothesis on~$\GG$, easily checked in all our examples, the Hom modules of $\spancat_\kk(\GG;\JJ)$ become particularly nice:
 
\begin{Prop} \label{Prop:ftfree-kHoms}
Let $(\GG;\JJ)$ be a spannable pair and consider $\spancat_\kk(\GG;\JJ)$, the associated $\kk$-linear category of spans. Then:
\begin{enumerate}[\rm(1)]
\item \label{it:Hom-freeness}
The Hom $\kk$-modules of $\spancat_\kk(\GG;\JJ)$ are all free, provided this holds: For every object $X$ of $\GG$ there exists an equivalence $X\simeq X_1 \sqcup \ldots \sqcup X_n$ to a coproduct of finitely many objects $X_i$ which are indecomposable with respect to coproducts; moreover, if $X\simeq X_1' \sqcup \ldots \sqcup X_m'$ is another such decomposition, then $n=m$ and there exist a permutation $\sigma\in \Sigma_n$ and equivalences $X'_i\simeq X_{\sigma(i)}$.
\item \label{it:Hom-ft}
 Assuming the hypothesis of \eqref{it:Hom-freeness} holds, the Hom $\kk$-modules  of $\spancat_\kk(\GG;\JJ)$ are all finitely generated (and free), provided that for every object $X$ the comma category $(\pih \GG)_{/X}= \pih( \GG_{/X})$ only has finitely many  $\sqcup$-indecomposable objects.
\end{enumerate}
\end{Prop}

\begin{proof}
Notice that it suffices to show the claims hold in the special case $\kk=\mathbb Z$, as the general case will immediately follow by extending scalars.

Assume the hypothesis of~\eqref{it:Hom-freeness}, and fix two objects $X,Y$. Choose a full set of representatives 
\[
s_\ell := \left[ X \overset{u_\ell}{\longleftarrow} S_\ell \overset{i_\ell}{\longrightarrow} Y  \right]  \,,\quad \ell\in \Lambda
\]
for those spans in $\spancat (\GG;\JJ)(X,Y)$ whose middle object $S_\ell$ is indecomposable (the latter property is invariant under equivalence of spans). We claim that $\{s_\ell\}_{\ell\in \Lambda}$ is a basis of $\spancat_\mathbb Z (\GG;\JJ)(X,Y)$.

By hypothesis~\eqref{it:Hom-freeness}, we may decompose the middle object of every span between $X$ and $Y$ into a coproduct of indecomposables, which induces a sum-decomposition of the span. Hence the~$s_\ell$ generate the abelian monoid $\spancat(\GG;\JJ)(X,Y)$.

To show they are linearly independent, we first prove that the abelian monoid $\spancat(\GG;\JJ)(X,Y)$ is cancellative, that is: if $s,s',t$ are elements such that $s+t=s'+t$ then $s=s'$.
Thus consider three parallel spans
\[ 
s=\left[ X \overset{u}{\longleftarrow}S  \overset{i}{\longrightarrow} Y \right] \,,\;\;
s'=\left[ X \overset{u'}{\longleftarrow}S'  \overset{i'}{\longrightarrow} Y \right] \textrm{ and } \;\;
t=\left[ X \overset{v}{\longleftarrow}T \overset{j}{\longrightarrow} Y \right] 
\]
and suppose we have an equivalence 
\begin{equation} \label{eq:cancellative-equiv}
\vcenter{
\xymatrix@R=.5em@C=4em{ 
& S \sqcup T \ar[dl]_{(u,v)} \ar[dr]^-{(i,j)} \ar[dd]^f_{\simeq} & \\
X \ar@{}[r]|{\;\;\;\;\alpha\SEcell}  && Y \ar@{}[l]|{\SWcell\beta\;\;\;\;} \\
& S' \sqcup T \ar[ul]^{(u',v)} \ar[ur]_{(i',j)} &
}} \,.
\end{equation}
Now let us decompose the spans $s+t$ and $s'+t$ as sums of indecomposable spans, using the representatives chosen above. In particular, we get decompositions 
\[
S\simeq \coprod_{\ell} S_\ell^{\sqcup n_\ell} 
\,, \quad
S'\simeq \coprod_{\ell} S_\ell^{\sqcup n'_\ell} 
\textrm{ and } \quad 
T\simeq \coprod_{\ell} S_\ell^{\sqcup m_\ell} 
\]
where the $n_\ell,n'_\ell,m_\ell$ are non-negative integers, almost all of which are zero.
By the extensivity of~$\GG$, the equivalence $f$ must be a coproduct of 1-cells between these factors.  By the uniqueness part of hypothesis~\eqref{it:Hom-freeness}, and since we picked one $S_\ell$ per equivalence class, we see that $f$ must be \emph{diagonal} with respect to them.

Now, it is \emph{a priori} possible  that, for a fixed~$\ell$, the equivalence $f$ matches some of the identical factors $S_\ell$ within $S$ and~$T$ or $S'$ and~$T$, but after composing $f$ if necessary with a self-equivalence of the target which permutes said factors, we may assume that it has the form $f\simeq f_1\sqcup f_2$ for two equivalences $f_1\colon S\overset{\sim}{\to} S'$ and $f_2\colon T\overset{\sim}{\to} T$. 
(Note also that two different $s_\ell\neq s_k$ may have equivalent middle objects~$S_\ell\simeq S_k$, but $f$ cannot match two such middle objects because then, by~\eqref{eq:cancellative-equiv}, $f$ together with (suitable components of) $\alpha$ and~$\beta$ would yield an equivalence $s_\ell = s_k$; but this can only hold if $\ell=k$.)
Then $\alpha = (\alpha_1,\alpha_2)$ and $\beta=(\beta_1,\beta_2)$ must decompose accordingly. In particular, we obtain from~\eqref{eq:cancellative-equiv} an equivalence
\[
\vcenter{
\xymatrix@R=.5em@C=4em{ 
& S  \ar[dl]_{u} \ar[dr]^-{i} \ar[dd]^{f_1}_{\simeq} & \\
X \ar@{}[r]|{\;\;\;\;\alpha_1\SEcell}  && Y \ar@{}[l]|{\SWcell\beta_1\;\;\;\;} \\
& S' \ar[ul]^{u'} \ar[ur]_{i'} &
}}
\]
showing that $s=s'$ in $\spancat(\GG;\JJ)(X,Y)$, as claimed.

Next, we show that the set $\{s_\ell\}_\ell$ is a basis of the abelian monoid $\spancat (\GG;\JJ)(X,Y)$ (\ie every element can be written as a sum of generators in a unique way). It will then easily follow from cancellativity that it is also a basis for the $\mathbb Z$-module $\spancat (\GG;\JJ)(X,Y)$.
Consider two finite sums yielding the same element:
\[
\sum_{\ell} n_\ell s_\ell = \sum_\ell m_\ell s_\ell
\]
with $n_\ell, m_\ell\in \mathbb N$ (and almost all zero). This means that there is an equivalence 
\[
\vcenter{
\xymatrix@R=.5em@C=4em{ 
& {\coprod_\ell S_\ell^{\sqcup n_\ell}}  \ar[dl]_{(u_\ell)} \ar[dr]^-{(i_\ell)_\ell} \ar[dd]^{f}_{\simeq} & \\
X \ar@{}[r]|{\;\;\;\;\alpha\SEcell}  && Y \ar@{}[l]|{\SWcell\beta\;\;\;\;} \\
& {\coprod_\ell S_\ell^{\sqcup m_\ell}} \ar[ul]^{(u_\ell)} \ar[ur]_{(i_\ell)_\ell} &
}}
\]
of spans.
By extensivity and~\eqref{it:Hom-freeness} as before, the equivalence $f$ down the middle must decompose diagonally as a coproduct of equivalences $f_\ell\colon S_\ell^{\sqcup n_\ell}\overset{\sim}{\to} S_\ell^{\sqcup m_\ell}$ at each~$\ell$ such that $\alpha\colon u_\ell \cong u_\ell f_\ell$ and $\beta\colon i_\ell\cong i_\ell f_\ell$. Moreover, each of these must match the factors one-on-one, since the $S_\ell$ are indecomposable. Hence we must have $n_\ell = m_\ell$ for all~$\ell$, as claimed.

For part~\eqref{it:Hom-ft}, the finiteness of the basis $\{s_\ell\}_\ell$ in an easy consequence of the hypothesis that each comma category $(\pih \GG)_{/X}$ (or equivalently each comma 2-category $\GG_{/X}$) only has finitely many indecomposable objects. 
\end{proof}

\begin{Def}[{Mackey functor; \cf \cite[Def.\,2.5.4]{BalmerDellAmbrogio18pp}}]
\label{Def:Mackey-fun}
Let $\spancat_\kk(\GG;\JJ)$ be the category of spans of \Cref{Cons:1-spans} for some spannable pair $(\GG;\JJ)$.
A \emph{$\kk$-linear Mackey functor for $\GG$ with respect to $\JJ$} is defined to be a $\kk$-linear (hence also additive, \ie direct-sum preserving) functor 
\[
M\colon \spancat_\kk(\GG;\JJ)\longrightarrow \kk\MMod
\]
to the abelian category of $\kk$-modules, and a morphism of Mackey functors is simply a natural transformation. We denote by 
\[
\Mackey_\kk(\GG;\JJ) := \Fun_\kk(\spancat_\kk(\GG;\JJ), \kk\MMod)
\]
the resulting $\kk$-linear category. By the general properties of $\kk$-linearization, we have a canonical additive functor $\bisetcat(\gpd)\to \bisetcat_\kk(\gpd)$ which induces an isomorphism
\begin{equation} \label{eq:iso+kk}
\Mackey_\kk(\GG;\JJ)  \stackrel{\sim}{\too} \Fun_+(\spancat(\GG;\JJ), \kk\MMod ) 
\end{equation}
of $\kk$-linear categories, identifying additive functors on $\spancat (\GG;\JJ)$ and $\kk$-linear functors on $\spancat_\kk(\GG;\JJ)$.
\end{Def}

\begin{Rem}
\label{Rem:Mackey-functor-for-GI}%
Explicitly, a Mackey functor for $(\GG;\JJ)$ as in \Cref{Def:Mackey-fun} (and via \eqref{eq:iso+kk}) consists of a $\kk$-module $M(X)$ for every object $X$ of~$\GG$, a $\kk$-homomorphism $u^*\colon M(Y)\to M(X)$ for every 1-morphism $u\in \GG(X,Y)$ (corresponding to the image of the span $[Y \overset{u}{\leftarrow} X = X]$) and a $\kk$-homomorphism $u_*\colon M(X)\to M(Y)$ if furthermore $u\in\JJ$ (the image of the span $[X = X \overset{u\;}{\rightarrow} Y]$). This data is subject to the following rules\,:
\begin{enumerate}[(1)]
\item
\emph{Functoriality:} $(\Id_X)^*= (\Id_X)_*= \id_{M(X)}$ and $(uv)^*= v^*u^*$ for all $X\in \GG$ and all composable 1-cells $u,v$; and also $(uv)_*=u_*v_*$ when $u,v$ belong to~$\JJ$.
\item 
\emph{Isomorphism invariance:} If two 1-cells $u\simeq v$ are isomorphic in the category $\GG(X,Y)$ then $u^*=v^*$, and also $u_*=v_*$ whenever $u,v$ belong to~$\JJ$.
\item
\emph{Additivity:} The canonical morphism $M(X_1\sqcup X_2)\isoto M(X_1)\oplus M(X_2)$ is an isomorphism for all $X_1,X_2\in\GG$, as well as $M(\emptyset)\overset{\sim}{\to}0$.
\item 
\emph{Mackey formula:} For every Mackey square in~$\GG$ with $i$ (and thus also~$j$) in~$\JJ$
\[\vcenter{\xymatrix@C=14pt@R=14pt{
& P \ar[dl]_-{v} \ar[dr]^-{j}
 \ar@{}[dd]|(.5){\isocell{\gamma}}
\\
X \ar[dr]_-{i}
&& Y \ar[dl]^-{u}
\\
&Z
}}
\]
we have $u^*i_*=j_*v^*\colon M(X)\to M(Y)$.
\end{enumerate}
\end{Rem}

\begin{Rem} \label{Rem:*notation}
Consistently with the previous remark, we may use the short-hand notation
\[
u^* := [ Y \overset{u}{\gets} X = X] \,, \quad \quad i_* := [Z =  Z \overset{i}{\to} W]
\]
for the `contravariant' and `covariant' spans associated to all 1-cells $u\in \GG$ and $i\in \JJ$. Note that every span can be written as $[\overset{u}{\gets} \,\overset{i}{\to}]= i_*u^*$. This defines two faithful functors 
\[
(-)^*\colon \pih \GG^{\op} \too \spancat(\GG;\JJ) 
\,, \quad \quad 
(-)_*\colon \pih \JJ \too \spancat(\GG;\JJ)
\]
which can be jointly used to formulate a simple universal property for the categories $\spancat(\GG;\JJ)$ and $\spancat_\kk(\GG;\JJ)$; see \eg \cite[A.4]{BalmerDellAmbrogio18pp}.
\end{Rem}

Here is the list of all examples of Mackey functors over spannable pairs~$(\GG;\JJ)$ that will be considered in this article.

\begin{Exa} 
\label{Exa:gpd}
Our first example is $\GG=\JJ=\gpd$, the 2-category of all finite groupoids, functors and natural isomorphisms. This is a spannable (2,1)-category, as one verifies easily (\cf \Cref{Rem:theG=Jcase}).
In this case, \Cref{Def:Mackey-fun} specializes to the notion of \emph{\textup(global\textup) Mackey functor} studied for example in~\cite{Ganter13pp} and~\cite{Nakaoka16}. 
We fully investigate the span category $\spancat_\kk(\gpd)$ in~\Cref{sec:presentation}.
\end{Exa}

The next two examples can also be understood as kinds of biset functors, as will be proved in \Cref{Cor:no-deflations}.

\begin{Exa} \label{Exa:gpdf}
Similarly, we may consider $\GG=\JJ=\groupoidf$, the 2-category of all finite groupoids, \emph{faithful} functors, and natural isomorphisms.
The resulting Mackey functors are also sometimes called \emph{global Mackey functors} (\cf \cite[Ex.\,3.2.6]{Bouc10}).
\end{Exa}

\begin{Exa} \label{Exa:gpd-gpdf}
As a proper example of a spannable \emph{pair}, consider $\GG=\gpd$ and $\JJ=\groupoidf$. The closure properties \eqref{it:J-closures}-\eqref{it:coprods} are easily checked.
The resulting Mackey functors are sometimes called \emph{inflation functors} (\cf \cite[Ex.\,3.2.5]{Bouc10}). 
\end{Exa}

\begin{Exa} \label{Exa:Dress_Mackey_ordinary}
For a fixed finite group~$G$, we may consider the comma 2-category $\GG=\JJ=\gpdG$ of $\groupoidf$ over~$G$. 
The resulting Mackey functors are the classical \emph{Mackey functors for~$G$}. See \Cref{sec:G-Mackey}.
\end{Exa}

\begin{Exa} \label{Exa:fused}
Again for a fixed $G$, we will consider a variant $\gpdGfuz$ of the comma 2-category of \Cref{Exa:Dress_Mackey_ordinary}, where we add all 2-morphisms which exist after forgetting to~$\gpd$. This imposes extra relations on the associated notion of Mackey functors, which turn out to be Bouc's \emph{fused Mackey functors for~$G$}. See \Cref{sec:G-Mackey}.
\end{Exa}

\begin{Rem}
Many more variations are possible and useful. In particular, we may also want to restrict the class of \emph{objects} of~$\GG$. For example, we may consider spans (of various kinds) only between groupoids consisting of coproducts of finite $p$-groups, or coproducts of subquotients of a single fixed group~$G$ (\cf \cite[\S\,9]{Webb00} and \cite[Part III]{Bouc10}).
 \end{Rem}

\section{Presenting spans of groupoids}
\label{sec:presentation}%

In this section we take a hands-on approach to the central object of this article, the category of spans of groupoids. In particular, we provide a presentation by generators and relations as a linear category. This presentation should look familiar to all practitioners of Mackey and biset functors, and indeed will be useful for the comparison results of the last section. The reader who is  easily bored by long lists of relations may read \Cref{Rem:comparison-biset} and then skip ahead to \Cref{sec:G-Mackey}. 
However our intention is not to be soporific, but rather to demonstrate how the usual `combinatorial' approach to Mackey and biset functors works just as well with spans of groupoids. 

Fix a commutative ring~$\kk$. The category we are interested in is 
\[
\spancat_\kk(\gpd) \,,
\]
the $\kk$-linear category of spans in finite groupoids, as defined in \Cref{Cons:1-spansk}.

\begin{Rem} \label{Rem:whySp(gpd)welldef}
\Cref{Cons:1-spansk} can be applied with $\GG:=\gpd$ because the latter is a spannable (2,1)-category in the sense of \Cref{Def:spannable}, as one can easily verify directly. In particular, $\gpd$ has well-behaved (strict) finite coproducts given by disjoint unions, and arbitrary (strict) Mackey squares given by the iso-comma construction of~\Cref{Exa:iso-comma-cats}. Moreover $\gpd$ satisfies the finiteness hypotheses of \Cref{Prop:ftfree-kHoms}, so that the Hom $\kk$-modules of $\spancat_\kk(\gpd)$ are all finitely generated and free.
Let us quickly recall from the previous section that the objects of $\spancat_\kk(\gpd)$ are the finite groupoids, and a morphism $\varphi \colon G\to H$ is a formal $\kk$-linear combination of equivalence classes of spans 
\[[G\xleftarrow{a} P \xrightarrow{b} H]\,,\]
 where $a,b$ are any two functors whose common source is a \emph{connected} groupoid~$P$. In practice, we also use non-connected $P$ but then we must identify $[G\xleftarrow{a} P \xrightarrow{b} H]$ with the sum $\sum_i [G\xleftarrow{a_i} P_i \xrightarrow{b_i} H]$, where $P=\coprod_i P_i$ is any disjoint union decomposition and $a_i,b_i$ are the corresponding components of $a,b$. 
 Here two spans $G\xleftarrow{a} P \xrightarrow{b} H$ and $G\xleftarrow{a'} P' \xrightarrow{b'} H$ are \emph{equivalent} iff there exists an equivalence $f\colon P\xrightarrow{\sim} P'$ making the diagram
\[
\xymatrix@R=.5em@C=4em{
& P \ar[ld]_-{a} \ar[dd]_-{\simeq}^-{f} \ar[rd]^-{b}
\\
G \ar@{}[r]|-{\simeq} && H \ar@{}[l]|-{\simeq}
\\
& P' \ar[lu]^-{a'} \ar[ru]_-{b'}
&
}
\]
commute up to isomorphisms of functors $a'f\simeq a$, $b'f\simeq b$.
Disjoint unions provide the direct sums of objects as well as, when taken at the middle object, the additive structure of each Hom module (see \Cref{Prop:1-spans}).
Composition in $\spancat_\kk(\gpd)$ is induced $\kk$-bilinearly from the composition of two spans $[G\xleftarrow{a} P \xrightarrow{b} H]$ and $[H \xleftarrow{c} Q \xrightarrow{d} K]$ obtained by constructing the iso-comma square~$(b/c)$. 
The groupoid $(b/c)$ will rarely be connected, hence the result of composing two spans is typically a \emph{sum} of classes of connected spans.
\end{Rem}

\begin{Def}[Global Mackey functor]
\label{Def:Mackeyfun-global}
The $\kk$-linear category of \emph{global Mackey functors} is obtained by specializing \Cref{Def:Mackey-fun} to $\GG=\JJ=\gpd$:
\[
\Mackey_\kk(\gpd) := \Fun_\kk(\spancat_\kk(\gpd), \kk\MMod) \,.
\]
This, in the form of \Cref{Rem:Mackey-functor-for-GI},  is precisely the same definition used in~\cite{Ganter13pp} and is equivalent to the definition of Mackey functor used in \cite{Nakaoka16} (via a biequivalence between $\gpd$ and Nakaoka's 2-category $\mathbb S$ of `sets with variable finite group action'; see \cite{Nakaoka16a}).
\end{Def}

\begin{Not}[Elementary spans] 
\label{Not:generators} 
We introduce five families of maps of~$\spancat_\kk(\gpd)$ which we call \emph{elementary spans}. These spans are all connected (\ie their middle object is connected) by virtue of only involving groups.
Let $i\colon H\hookrightarrow G$ denote a subgroup inclusion, $p\colon G\onto G/N$ a quotient homomorphism, and $f \colon G\stackrel{\sim}{\to} G'$ an isomorphism  between finite groups. Then we write
\begin{align*}
& \Res^G_H := i^* = [G \stackrel{i}{\hookleftarrow} H = H]  && \Ind_H^G := i_* = [H = H \stackrel{i}{\hookrightarrow} G] \\
& \Infl_{G/N}^G := p^* = [G/N \stackrel{p}{\twoheadleftarrow} G = G] && \Defl^G_{G/N} := p_* = [G = G \stackrel{p}{\onto} G/N] 
\end{align*}
\vspace{-12pt}
\[
\Iso(f) := (f^{-1})^* = f_* = [G \xleftarrow{f^{-1}} G' = G'] =  [G = G \xrightarrow{f} G'] 
\]
for the corresponding equivalence classes of spans in~$\spancat_\kk(\gpd)$ and call them, respectively, \emph{restrictions}, \emph{inductions}, \emph{deflations}, \emph{inflations} and \emph{isomorphisms}. 
\end{Not}

\begin{Thm}[A presentation of~$\spancat_\kk(\gpd)$]
\label{Thm:pres-Span}
As an additive $\kk$-linear category, $\spancat_\kk(\gpd)$ is generated by the elementary spans of \Cref{Not:generators}; this means that every object  is a direct sum of finite groups, and that every morphism is a matrix of maps between groups, each of which can be obtained as a $\kk$-linear combination of composites of elementary spans. 
Moreover, the ideal of all relations between the maps of $\spancat_\kk(\gpd)$ is generated by the following three families:
\begin{enumerate}[\rm1.] 
\setcounter{enumi}{-1}
\smallbreak
\item \emph{Triviality relations:}
\begin{enumerate}[\rm(a)] 
\item  For every group~$G$: 
\[
\Res^G_G = \id_G \,, \quad \Ind^G_G = \id_G \,, \quad \Defl^G_{G/1} = \id_G, \quad \Infl^G_{G/1} = \id_G \,.
\]
\item For every \emph{inner} automorphism $f \colon G\xrightarrow{\sim}G$:
\[
\Iso(f)  = \id_G \,.
\]
\end{enumerate}
\smallbreak
\item \emph{Transitivity relations:}
\begin{enumerate}[\rm(a)]
\item For all subgroups $K\leq H \leq G$:
\[
\Res^H_K \circ \Res^G_H = \Res^G_K , \quad\quad \Ind_H^G \circ \Ind_K^H = \Ind_K^G  \,.
\]
\item For any two composable isomorphisms $G \xrightarrow{f} G'  \xrightarrow{f'} G''$:
\[
\Iso(f') \circ \Iso(f) = \Iso(f'\circ f)  \,.
\]
\item For any two normal subgroups $N,M\trianglelefteq G$ with $N\leq M$:
\[
\Infl^G_{G/N} \circ \Infl^{G/N}_{G/M} = \Infl^G_{G/M} \,, \quad\quad \Defl^{G/N}_{G/M} \circ \Defl^G_{G/N} = \Defl^G_{G/M}  \,.
\]
\end{enumerate}
\smallbreak
\item \emph{Commutativity relations:}
\begin{enumerate}[\rm(a)]
\item For a subgroup $K\leq G$ and an isomorphism $f\colon G\overset{\sim}{\to} G'$, where $\tilde f\colon K\overset{\sim}{\to}f(K)$ denotes the isomorphism induced by~$f$:
\begin{align*}
\Iso(\tilde f) \circ \Res^G_K &= \Res^{G'}_{f(K)} \circ \Iso(f) \,, \\
\Iso(f) \circ \Ind^G_K &= \Ind^{G'}_{f(K)} \circ \Iso(\tilde f)  \,.
\end{align*}
\item For a normal subgroup $N\trianglelefteq G$ and an isomorphism $f\colon G\xrightarrow{\sim}G'$, where $\overline f\colon G/N \xrightarrow{\sim} G'/f(N)$ is the isomorphism induced by~$f$:
\begin{align*}
\Iso(\overline{f}) \circ \Defl^G_{G/N} &= \Defl^{G'}_{G'/f(N)} \circ \Iso(f)    \,,  \\
\Iso(f)\circ \Infl^G_{G/N} &= \Infl^{G'}_{G'/f(N)} \circ \Iso(\overline{f})  \,.
\end{align*}
\item \emph{(The Mackey formula.)} For any two subgroups $H,K\leq G$,  where $[H\backslash G/ K]$ denotes any full set of representatives of the double cosets $H\backslash G/ K$ and $c_x\colon H^x\cap K \xrightarrow{\sim} H\cap{}^x\!K$ is the conjugation isomorphism $g\mapsto {}^xg = xgx^{-1}$:
\[
\Res^G_H \circ \Ind^G_K = \sum_{x \in [H\backslash G/ K]} \Ind^H_{H\cap {}^x\!K} \circ \Iso(c_x) \circ \Res^K_{H^x\cap K}  \,.
\]
\item For any two normal subgroups $M,N\trianglelefteq G$ such that $M\cap N= 1$
we have:
\[
\Defl^G_{G/N} \circ \Infl^G_{G/M} = \Infl^{G/N}_{G/MN} \circ \Defl^{G/M}_{G/MN}  \,.
\]

\item 
For a subgroup $H\leq G$ and a normal subgroup $N\trianglelefteq G$, where $f\colon H/H\cap N \xrightarrow{\sim} HN/N$ denotes the canonical isomorphism:
\begin{align*}
\Defl^G_{G/N} \circ \Ind^G_H &= \Ind^{G/N}_{HN/N} \circ \Iso(f) \circ \Defl^H_{H/H\cap N}  \,, \\
\Res^G_H \circ \Infl^G_{G/N} &= \Infl^H_{H/H\cap N} \circ \Iso(f^{-1}) \circ \Res^{G/N}_{HN/N}  \,.
\end{align*}

\item
For a subgroup $H\leq G$ and a normal subgroup $N\trianglelefteq G$ such that $N\leq H$:
\begin{align*}
\Res^{G/N}_{H/N}\circ \Defl^G_{G/N} &= \Defl^H_{H/N} \circ \Res^G_{H} \,, \\
\Ind^G_H \circ \Infl^H_{H/N}  &= \Infl^G_{G/N} \circ \Ind^{G/N}_{H/N}  \,.
\end{align*}
\end{enumerate}
\end{enumerate}
\end{Thm}

\begin{Rem}
As commonly done, we have identified quotients such as $G/1= G$ in 0.(a) or $(G/M)/MN=G/MN$ in~2.(d), \ie we have omitted some (very canonical) isomorphism maps from the presentation. This should not cause any problems.
\end{Rem}

\begin{Rem} \label{Rem:comparison-biset}
The presentation in \Cref{Thm:pres-Span} is almost identical to the presentation provided in \cite[\S\,1.1.3]{Bouc10} for the biset category; for ease of comparison, we have kept the same notations and numbering, except for changing `3.'\ to~`0.'. Indeed, the \emph{only difference} lies in the family~2.(d), where Bouc's 2.(d) is stronger than ours because it holds for any pair $M,N\trianglelefteq G$ of normal subgroups, with no condition on their intersection.
Spans and bisets will be compared in~\Cref{sec:bisets}, but the conclusion should already be clear: bisets can be obtained from spans simply by requiring 2.(d) to hold for all normal subgroups $M,N \trianglelefteq G$, not only those with $M\cap N=1$. We will see that it actually suffices to replace the above 2.(d) by the following \emph{deflativity relation}
\[\Defl^G_{G/N} \circ \Infl^G_{G/N} = \id_{G/N}\]
for all normal subgroups $N\trianglelefteq G$. 
\end{Rem}

\begin{Rem} \label{Rem:pullbacks-and-2d}
The condition $M\cap N= 1$ in relation 2.(d) is equivalent to the square 
\[
\vcenter{
\xymatrix@C=8pt@R=8pt{ 
& G \ar@{->>}[dl] \ar@{->>}[dr] & \\
G/M \ar@{->>}[dr] && G/N \ar@{->>}[dl] \\
& G/MN&
}}
\]
of quotient maps being a pull-back of groups. Indeed, the comparison map to the pull-back  $f\colon G\to G/M \times_{G/MN} G/N=:P$, $g\mapsto f(g)=(gM,gN)$, is injective iff $M\cap N= 1$, and is always automatically surjective: Given $x=(g_1M, g_2N)\in P$, we can write $g_1m_1n_1 = g_2m_2n_2$ for some $m_1,m_2\in M$ and $n_1,n_2\in N$, hence $g_2= g_1m_1n_1n_2^{-1}m_2^{-1}\in g_1mN$ for some $m\in M$ (recall $N$ is normal), hence $g_2N=g_1mN$ and of course also $g_1M=g_1mM$, hence $x= f(g_1m)$, showing surjectivity.
\end{Rem}

To prove the theorem, we can first reduce the problem from groupoids to groups.

\begin{Lem} \label{Lem:add-closure-spans}
Let $\spancat_\kk(\group)\subset \spancat_\kk(\gpd)$ denote (by a slight abuse of notation) the full subcategory whose objects are all finite groups. Then $\spancat_\kk(\gpd)$ is the additive hull of $\spancat_\kk(\group)$, that is: every object of $\spancat_\kk(\gpd)$ is (isomorphic to) a direct sum of objects of $\spancat_\kk(\group)$, and every map decomposes as a matrix of maps in $\spancat_\kk(\group)$.
\end{Lem}

\begin{proof}
Every finite groupoid is equivalent to a finite disjoint union of groups, and disjoint unions provide the direct sums ($=$~biproducts) in~$\spancat_\kk(\gpd)$. The rest follows immediately from the basic properties of direct sums (\S\,\ref{parag:k-lin}).
\end{proof}

\begin{proof}[Proof of \Cref{Thm:pres-Span}]
Since the relations listed in \Cref{Thm:pres-Span} only involve maps between groups, it will suffice to show that $\spancat_\kk(\group)$ is generated \emph{as a $\kk$-linear category} (without using direct sums) by the same generators and relations. The theorem will then follow from \Cref{Lem:add-closure-spans} by taking additive hulls.

Let $\mathcal F$ denote the free $\kk$-linear category generated by the elementary spans and let $\mathcal J$ denote the $\kk$-linear categorical ideal of maps in $\mathcal F$ generated by all the relations in the theorem (\ie by the corresponding differences, of course). Thus by construction we have a $\kk$-linear functor $\Phi\colon \mathcal F\to \mathcal \spancat_\kk(\group)$, and we must show that it induces an equivalence
\[
\overline{\Phi}\colon \mathcal F/\mathcal J\xrightarrow{\sim} \spancat_\kk(\group)\,.
\]

Let us prove that the factorization $\overline{\Phi}$ exists, that is, that the three families of relations are all satisfied in~$\spancat_\kk(\group)$.
This is a straightforward verification, which involves constructing many iso-comma squares in order to compute the composites of various pairs of elementary spans $[a,b]$ and~$[c,d]$:
\begin{equation} \label{eq:comp-Mackey-square}
\vcenter{
\xymatrix@L=1pt@R=14pt@C=18pt{
&& \ar[ld]_-{\tilde c} \ar[rd]^-{\tilde b} 
&&
\\
& \ar[ld]^-a \ar[rd]_-{b} \ar@{}[rr]|{\overset{\sim}{\Ecell}} && \ar[ld]^-{c} \ar[rd]_-d &
\\
 \ar@{..>}[rr] && \ar@{..>}[rr] &&
}}
\end{equation}
Most relations are immediately obtained by composing spans where either $b$ or $c$ is an isomorphism (\eg an identity map), in which case the iso-comma square can be replaced by an equivalent \emph{commutative} square where two parallel sides are isomorphisms (\eg identities); see \cite[Rem.\,2.1.9]{BalmerDellAmbrogio18pp} if necessary.    
This takes care of all transitivity relations 1.\ and of the commutativity relations 2.(a),(b) as well as 2.(e); for the latter, we may also use the equality of maps
$(H \hookrightarrow G\twoheadrightarrow G/N)=(H \twoheadrightarrow H/H\cap N \cong HN/N\hookrightarrow G/N)$.

 The triviality relations 0.(a) are trivial and 0.(b) can be checked as follows. Recall that we are viewing groups as one-object groupoids and homomorphisms as functors (\S\,\ref{parag:groupoids}-\ref{parag:bicats}), so that if the isomorphism $f$ is $c_x\colon G\xrightarrow{\sim} G$, $g \mapsto xgx^{-1}$, conjugation by an element~$x\in G$, then there is a natural isomorphism $\gamma_x\colon \Id_G\Rightarrow c_x$ as in \Cref{Not:gamma_x}. This yields an equivalence
\[
\vcenter{
\xymatrix@R=.5em@C=4em{ 
& G \ar@{=}[dl] \ar[dr]^-{c_x} \ar[dd]^{c_x} & \\
G \ar@{}[r]|{\;\;\;\;\gamma_x\SEcell}  && G  \\
& G \ar@{=}[ul] \ar@{=}[ur] &
}}
\]
between the spans $\Iso(c_x)=[G=G\xrightarrow{c_x} G]$ and $\id_G=[G=G=G]$.

For the remaining relations, we will compose elementary spans where neither $b$ nor $c$ is an isomorphism, \ie where both are non-trivial inclusions or quotient maps 
(and hence where $a$ and $d$ are identities, since these are elementary spans). We will deduce various commutativity relations of the form $c^*b_* = \tilde b_* \tilde c^*$ (in the notation $(-)^*$ and $(-)_*$ of \Cref{Rem:*notation} applied to \eqref{eq:comp-Mackey-square}).

\begin{Lem}  \label{Lem:pullback-vs-isocomma}
If either $b$ or $c$ is a surjection of groups, then the iso-comma square and the pullback square over $\xrightarrow{b}\xleftarrow{c}$ are equivalent.
More precisely, the canonical comparison functor $w$ with components $\langle \pr_1, \pr_2 , \id \rangle$ (see \S\,\ref{parag:isocomma})
\begin{equation*}
\vcenter{
\xymatrix@C=14pt@R=14pt{
&B\times_DC \ar@/_3ex/[ddl]_{\pr_1} \ar@/^3ex/[ddr]^{\pr_2} \ar[d]^-{\simeq}_-{w} & \\
& (b/c) \ar[dl]_{\tilde c} \ar[dr]^-{\tilde b} & \\
B \ar[dr]_b \ar@{}[rr]|{\oEcell{\sim}} && C \ar[dl]^c \\
&D &
}}
\end{equation*}
is an equivalence. 
\end{Lem}

\begin{proof} 
Indeed, say that $c$ is surjective (the case of $b$ is similar). Then $w$ sends the unique object $\bullet$ of the pullback group $B\times_D C$ to $(\bullet_B,\bullet_C, \id_{\bullet_D})$, and every object $(\bullet,\bullet,\delta)\in \Obj (b/c)$ is isomorphic to $(\bullet,\bullet,\id)$ via a map $(\id, \gamma)$ for any $\gamma\in C$ such that $c(\gamma)=\delta$.
It is immediate to see that $w$ is also fully faithful.

 (More generally, pull-backs and iso-commas along a functor $f$ are equivalent if $f$ has the invertible path lifting property, \ie if it is a fibration in the canonical model structure; see \cite{JoyalStreet93}.)
\end{proof}

This lemma takes care of most remaining cases. 
For instance, a pull-back as in \Cref{Rem:pullbacks-and-2d} is equivalent to the iso-comma square over $G/M\to G/MN \gets G/N$ and the commutivity relation stated in 2.(d) then follows immediately. 
This was the case where $b,c$ are both quotient maps; the cases where one is a quotient and the other an inclusion prove the relations~2.(f). (For the latter, use also that the square
\[
\vcenter{
\xymatrix@!C=8pt@R=9pt{ 
& H \ar@{->>}[dl] \ar[dr]^{\incl} & \\
H/N \ar[dr]_{\incl} && G \ar@{->>}[dl] \\
& G/N&
}}
\]
is automatically a pullback.)

The only remaining relation is the Mackey formula 2.(c), which corresponds to the case where both $b=:j$ and $c=:i$ are inclusions. 
This is proved by decomposing the relevant iso-comma square by the equivalence $v:=\langle ( \incl), ( \incl c_x) , (\gamma_x) \rangle$ below which, concretely, sends the object $\bullet_{H^x\cap K}$ to $(\bullet_K,\bullet_H,x)\in \Obj (j/i)$ and sends $h^x=k\in H^x\cap K$ of the component at $[x]$ to the map $(k,h)$ of~$(j/i)$:
\begin{equation*}
\vcenter{
\xymatrix@C=14pt@R=18pt{
&{\coprod\limits_{[x]\in H\bs G/K}} \!\!\! H^{x}\!\cap K 
 \ar@/_3ex/[ddl]_{(\incl)}
  \ar@/^3ex/[ddr]^{( \incl \,\circ\, c_x )}
   \ar[d]^-{\simeq}_-{v} & \\
& (j/i) \ar[dl]_{} \ar[dr]^-{} & \\
K \ar[dr]_{j \, = \, \incl\;\;}
 \ar@{}[rr]|{\oEcell{\sim}} &&
  H \ar[dl]^{\;\; i \,= \,\incl} \\
&G &
}}
\end{equation*}
Here $\gamma_x$ denotes the natural isomorphism 
\[
\vcenter{
\xymatrix@C=12pt@R=14pt{ 
& H^x\cap K \ar[dl]_\incl \ar[dr]^{\incl \, c_x} & \\
K \ar[dr]_j \ar@{}[rr]|{\oEcell{\gamma_x}} && H \ar[dl]^i \\
& G&
}}
\]
whose sole component is $x\in G$ 
(\cf \cite[Rem.\,2.2.7]{BalmerDellAmbrogio18pp}).

Thus we have proved so far that the functor $\overline{\Phi}\colon \mathcal F/\mathcal J\to \spancat_\kk(\group)$ is well-defined, and we must show that it is an equivalence. 
As $\overline{\Phi}$ is the identity on objects, it only remains to show that it is fully faithful. It is certainly full, because every span 
\begin{equation} \label{eq:typical-span}
\vcenter{
\xymatrix@L=1pt@R=14pt@C=18pt{
& S \ar[ld]_-b \ar[rd]^-{a} &
\\
H \ar@{..>}[rr] &&  G
}}
\end{equation}
of group homomorphisms admits the decomposition
\[
\xymatrix@R=14pt@C=18pt{
&&& 
S
 \ar@{->>}[dl] 
 \ar@{->>}[dr]  
 \ar@/_1ex/[dlll]_-b 
 \ar@/^1ex/[drrr]^-a
&&& \\
H & \;\;D \ar@{>->}[l] & S/N \ar[l]^-{\sim}_-{\ell} && S/M \ar[r]_-\sim^-{f} & B\;\; \ar@{>->}[r] & G
}
\]
where $B:=\Img(a)$, $D:=\Img(b)$, $M:=\Ker(a)$ and $N:=\Ker(b)$, and therefore can be written as the following composite of six elementary spans:
\begin{equation} \label{eq:can-form}
\Ind^G_B \circ \Iso(f) \circ \Defl^S_{S/M} \circ \Infl^S_{S/N} \circ \Iso(\ell^{-1}) \circ \Res^H_D \,.
\end{equation}
That is, the image in $\spancat_\kk(\group)$ of this formal composite is \eqref{eq:typical-span} by construction.

In order to prove faithfulness, let us first notice that the relations of $\mathcal J$ allow us to transform an arbitrary (composable) finite string of elementary spans in $\mathcal F/\mathcal J$ into a linear combination of strings of the form~\eqref{eq:can-form}, so that the latter form a $\kk$-linear generating set. Indeed, the relations 2.(a),(c),(e),(f) let us bring all induction maps to the left of all other elementary spans, where they can be combined in a single $\Ind$ by relation~1.(a).  Similarly, 2.(d),(f) let us drag any deflations to the left of all inflations and restrictions, where they can be combined into a single $\Defl$ by 1.(c). Inflations can be brought to the left of restrictions by 2.(e) and combined into one by~1.(c).
Finally, by 2.(a),(b), isomorphisms end up cumulating in only two spots, as indicated, where they can be combined by~1.(b).

Consider now two such length-six formal strings, one as in \eqref{eq:can-form} and one with primed notations $a',b',S', B',D',M',N',f',\ell'$.
Assume that they have the same image in $\spancat_\kk(\group)$, \ie that there exists an equivalence of spans as follows:
\[
\vcenter{
\xymatrix@R=.5em@C=4em{ 
& S \ar[dl]_b \ar[dr]^-a \ar[dd]^s_\simeq & \\
H \ar@{}[r]|{\;\;\;\;\beta\SEcell}  && G \ar@{}[l]|{\SWcell\alpha\;\;\;\;} \\
& S' \ar[ul]^{b'} \ar[ur]_{a'} &
}}
\]
Here $\alpha\colon a\Rightarrow a's$ and $\beta\colon b\Rightarrow b's$ are some natural isomorphisms, which (since $S$ is a group) are simply given by two elements $x\in G$ and $y\in H$ such that $a's={}^xa$ and $b's={}^yb$. Moreover, $s$ is a group isomorphism.
All this information can be reorganized into a commutative diagram of groups
\begin{equation} \label{eq:standard-form-almost-uniqueness}
\vcenter{
\xymatrix{
H &
 \;\;D \ar@{>->}[l] &
  S/N \ar[l]^-\simeq_-{f} \ar[d]^\simeq_{s_N} &
   S \ar@{->>}[l] \ar@{->>}[r] \ar[d]_s^\simeq &
    S/M \ar[d]^\simeq_{s_M} \ar[r]_-\simeq^-{\ell} &
     B \;\;\ar@{>->}[r] & G \\
H \ar@{<-}[u]_{c_y}^\simeq &
 \;\;D' \ar@{>->}[l] \ar@{<-}[u]^\simeq_{c_y} &
  S'/N' \ar[l]_-{\simeq}^-{f'} &
   S' \ar@{->>}[l] \ar@{->>}[r] &
    S'/M' \ar[r]^-{\simeq}_-{\ell'} &
     B'\;\; \ar@{<-}[u]^{c_x}_\simeq \ar@{>->}[r] &
      G \ar@{<-}[u]_\simeq^{c_x}
}}
\end{equation}
where all maps are either inclusions, quotients or isomorphisms; here $c_x$ and $c_y$ denote conjugation isomorphisms by $x$ and~$y$, respectively, and $s_M,s_N$ are the isomorphisms induced by $s$ on the quotients (indeed $M'=s(M)$ and $N'=s(N)$). Each square corresponds to a relation of~$\mathcal J$, with which we compute:
\begin{align*}
& \phantom{ = \;\;} \Ind^G_{B'} \Iso(\ell') \Defl^{S'}_{S'/M'} \Infl^{S'}_{S'/N'} \Iso(f'^{-1}) \Res^H_{D'}  
 &&  \\
&= \, \Iso(c_x^{-1}) \Ind^G_{B'} \Iso(\ell') \Defl^{S'}_{S'/M'} \Infl^{S'}_{S'/N'} \Iso(f'^{-1}) \Res^H_{D'} \Iso(c_y) 
 && \textrm{by 0.(b)} \\
&= \, \Ind^G_{B} \Iso(c_x^{-1}) \Iso(\ell') \Defl^{S'}_{S'/M'} \Infl^{S'}_{S'/N'} \Iso(f'^{-1}) \Iso(c_y) \Res^H_{D}  
  && \textrm{by 2.(a)}\\
&= \, \Ind^G_{B} \Iso(\ell) \Iso(s_M^{-1}) \Defl^{S'}_{S'/M'} \Infl^{S'}_{S'/N'} \Iso(s_N) \Iso(f^{-1}) \Res^H_{D}  
  && \textrm{by 1.(b)}\\
&= \, \Ind^G_{B} \Iso(\ell) \Defl^{S}_{S/M} \Iso(s^{-1}) \Iso(s) \Infl^{S}_{S/N} \Iso(f^{-1}) \Res^H_{D}
  && \textrm{by 2.(b)} \\
  &= \, \Ind^G_{B} \Iso(\ell) \Defl^{S}_{S/M} \Infl^{S}_{S/N} \Iso(f^{-1}) \Res^H_{D}
  && \textrm{by 1.(b).}
\end{align*}
Hence the two given parallel length-six composites are already equal in~$\mathcal F/\mathcal J$. 

Thus the $\kk$-linear map $\overline{\Phi}\colon \mathcal F/\mathcal J(H,G)\to \spancat_\kk(\group)(H,G)$ can be restricted to a bijection between a generating set of  $\mathcal F/\mathcal J(H,G)$ and a generating set of $\spancat_\kk(\group)(H,G)$. As the latter $\kk$-module is free (\Cref{Rem:whySp(gpd)welldef}), this map is an isomorphism.
\end{proof}

\section{$G$-sets vs groupoids}
\label{sec:G-Mackey}%

Let us fix a finite group $G$ throughout this section. We begin by  recalling from \cite[App.\,B]{BalmerDellAmbrogio18pp} how to use groupoids in order to capture the classical notion of Mackey functor for~$G$.

First of all, recall that the category of $\kk$-linear Mackey functors for $G$, which first (implicitly) appeared in~\cite{Green71}, can be equivalently defined as the functor category over spans of $G$-sets:
\begin{equation} \label{def:Mack(G)}
\Mackey_\kk(G) := \Fun_\kk( \spancat_\kk (G\sset) , \kk\MMod)\,.
\end{equation}
This approach is due to Dress~\cite{Dress73} and Lindner~\cite{Lindner76}. 

\begin{Rem}
Here $G\sset$ denotes the ordinary category of finite left $G$-sets, seen as a discrete 2-category. It makes sense to apply to $G\sset$ the span category construction $\spancat_\kk(-)$ of \Cref{Cons:1-spansk} because it is a spannable 2-category as in \Cref{Def:spannable}. Indeed, it is a (2,1)-category (like any discrete 2-category) which is extensive (like any elementary topos) and which moreover has arbitrary Mackey squares, because in an ordinary category the latter are the same thing as the usual pullbacks.
\end{Rem}

\begin{Def}[Groupoids faithfully embedded in~$G$] \label{Def:gpdG}
Let $\groupoidf$ denote the 2-category of finite groupoids, \emph{faithful} functors between them, and natural transformations.
We will consider the comma 2-category $\gpdG$ of $\groupoidf$ over~$G$ (as in~\S\,\ref{parag:comma2cat}), and call its objects $(H,i_H\colon H\to G)$ \emph{groupoids faithfully embedded in~$G$}. 
\end{Def}

\begin{Def}[Transport groupoid]
\label{Def:transport-gpd}
The \emph{transport groupoid} $G\ltimes X$ of a $G$-set $X$ is the groupoid with set of objects $\Obj (G\ltimes X):=X$ and where an arrow $x\to y$ is a pair $(g,x)\in G\times X$ such that $gx=y$ (we will occasionally also write $g\colon x\to y$ for simplicity). Composition is induced by the multiplication in $G$ via $(h,y)(g,x):=(hg,x)$. 
A $G$-equivariant map $f\colon X\to Y$ induces a faithful (!) functor $G\ltimes f\colon G\ltimes X\to G\ltimes Y$ which sends $x$ to $f(x)$ and $(g,x)$ to $(g,f(x))$. 
The transport groupoid comes equipped with a faithful functor $\pi_X\colon G\ltimes X\to G$ mapping $x\mapsto \bullet$ and $(g,x)\mapsto g$. 
(Note that the latter is just $G\ltimes (X\twoheadrightarrow G/G)$ followed by the obvious isomorphism $G\ltimes G/G = G$.)
\end{Def}

\begin{Prop}[{\cite[Prop.\,B.08]{BalmerDellAmbrogio18pp}}]
\label{Prop:bieq-Gsets-gpdG}
The transport groupoid of \Cref{Def:transport-gpd} defines a 2-functor (strict pseudo-functor) $G\ltimes - \colon G\sset \to \gpdG$, $X\mapsto (G\ltimes X, \pi_X)$, which is a biequivalence between the discrete 2-category of $G$-sets and the comma 2-category of groupoids faithfully embedded in~$G$.
\qed
\end{Prop}

Like any biequivalence, this one preserves Mackey squares and generally shows that $\gpdG$, just like $G\sset$, is a spannable 2-category. One can also check the latter by hand, for instance the fact that the iso-comma squares of $\gpd$ induce iso-comma squares in~$\gpdG$. 

By \Cref{Prop:functoriality-of-spancat}, if we  apply the span construction of \Cref{Cons:1-spansk} to the biequivalence of \Cref{Prop:bieq-Gsets-gpdG} we get:

\begin{Cor} \label{Cor:span-equiv-Gsets-gpdG}
The transport groupoid induces an equivalence 
\[ \spancat_\kk(G\sset)\overset{\sim}{\to} \spancat_\kk(\gpdG) \]
of $\kk$-linear categories of spans. \qed
\end{Cor}

Finally, by taking categories of $\kk$-linear functors:

\begin{Cor} \label{Cor:equivalence-MackeyG}
The transport groupoid induces a $\kk$-linear equivalence 
\[
\Mackey_\kk(G) \overset{\sim}{\gets} \Mackey_\kk(\gpdG)
\]
between the classical category of Mackey functors for $G$ and the category of generalized Mackey functors $\Mackey_\kk(\GG)$ (\Cref{Def:Mackey-fun}) applied to the 2-category $\GG=\gpdG$ of groupoids faithfully embedded in~$G$ (\Cref{Def:gpdG}). \qed
\end{Cor}

\begin{center}$***$\end{center}

For the remainder of this section, we adapt the above ideas in order to capture the \emph{fused} Mackey functors of~\cite{Bouc15} (whose definition will be recalled later). 

First, we `correct' the fact that the forgetful functor $\gpdG\to \gpd$ is not 2-full:

\begin{Def}[The fused comma category]
\label{Def:gpdG-fused}%
\index{$\gpdGfuz$}%
We consider the following variant $\gpdGfuz$ of the comma 2-category $\gpdG$ considered above, which we  call the 2-category of \emph{fused groupoids embedded into~$G$}. The 2-category $\gpdGfuz$ has the same 0-cells and 1-cells as $\gpdG$,
but has the larger class of 2-cells obtained by ignoring the compatibility requirement with the embeddings. Explicitly, the 0-cells of $\gpdGfuz$ are finite groupoids embedded into~$G$, \ie pairs $(H,i_H)$ with $i_H\colon H\into G$ faithful. The 1-cells $(H,i_H)\to (K,i_K)$ are pairs $(u\colon H\into K , \alpha_u\colon i_K u\isoEcell i_H)$. The new 2-cells $(u\colon H\to K , \alpha_u)\Rightarrow(v\colon H\to K , \alpha_v)$ are simply 2-cells $\alpha\colon u\Rightarrow v$ in~$\gpd$, without further condition. 
There are obvious inclusion and forgetful 2-functors
\[
\gpdG \hookrightarrow \gpdGfuz \to \gpd
\]
which are the identity on objects.
\end{Def}

We now introduce a 2-category which is to $G\sset$ what $\gpdGfuz$ is to~$\gpdG$:

\begin{Def}[The 2-category of fused $G$-sets]
\label{Def:G-set-fuz}%
Let $G^c$ denote the set $G$ equipped with the conjugation left $G$-action $(g,x)\mapsto {}^gx=gxg^{-1}$.
We define a 2-category $G\ssetfuz$ \emph{of fused $G$-sets}, whose 0-cells and 1-cells are those of $G\sset$, namely finite $G$-sets and $G$-equivariant maps, and whose 2-cells $\tau\colon f_1\Rightarrow f_2\colon X\to Y$ are given by $G$-maps $\tau\colon X\to G^c$ such that 
\[
\tau*f_1=f_2,
\]
where the notation means $\tau(x)\cdot f_1(x)=f_2(x)$ for all~$x\in X$. (Such $G$-maps $\tau\colon X\to G^c$ are called \emph{twisting maps}.) Vertical composition of 2-cells in~$G\ssetfuz$ is defined by multiplication in~$G$, that is, $(\tau'\cdot \tau)(x)=\tau'(x)\cdot \tau(x)$ for all~$x\in X$, and horizontal composition of 2-cells
\begin{equation} \label{eq:hor-comp}
\vcenter{
\xymatrix{
X \ar@/^3ex/[r]^-{f_1} \ar@/_3ex/[r]_-{f_2} \ar@{}[r]|{\Scell\;\tau} &
 Y \ar@/^3ex/[r]^-{f_3} \ar@/_3ex/[r]_-{f_4} \ar@{}[r]|{\Scell\;\sigma} &
 Z
}}
\quad = \quad
\vcenter{
\xymatrix{
X \ar@/^3ex/[r]^-{f_3\circ f_1} \ar@/_3ex/[r]_-{f_4\circ f_2} \ar@{}[r]|{\Scell\;\sigma\circ \tau} &
 Y
}}
\end{equation}
is given by $(\sigma \circ \tau)(x) = \tau(x)\cdot \sigma(f_1(x))$ for all $x\in X$. The identity 2-cell $\id_f$ of a 1-cell $f\colon X\to Y$ is given by the constant map $X\to G^c$, $\id_f (x)=e$ for all $x\in X$.
\end{Def}

\begin{Prop} \label{Prop:G-set-fuz}
The construction in \Cref{Def:G-set-fuz} yields a well-defined (2,1)-category $G\ssetfuz$ whose 1-truncation $\pih(G\ssetfuz)$ is equal to Bouc's ordinary \emph{category of fused $G$-sets}, $G\ssetfused$, as defined in \cite{Bouc15}.
Moreover, this (2,1)-category has Mackey squares for all cospans, provided by the usual pullback squares of $G$-sets, and is in fact a spannable (2,1)-category (\Cref{Def:spannable}).
\end{Prop}

\begin{proof}
The first claim is a direct verification from the definitions which we leave to the reader (\cf \Cref{Rem:2-cat-check}).
Bouc \cite{Bouc15} defines $G\ssetfused$ as the quotient category of $G\sset$ obtained by identifying any two parallel maps $f_1,f_2$ such that $\tau*f_1=f_2$ for some twisting map~$\tau$ (the resulting relation turns out to be a congruence); clearly, this is precisely the same as the truncated category $\pih(G\ssetfuz)$.

Let us verify the claim about Mackey squares. Consider a pullback square of $G$-sets and view it inside $G\ssetfuz$:
\begin{equation*}
\vcenter{
\xymatrix@C=14pt@R=14pt{
& X\times_ZY \ar[dl]_{p} \ar[dr]^-{q} & \\
X \ar[dr]_f \ar@{}[rr]|{\underset{\sim}{\oEcell{\id}}} && Y \ar[dl]^g \\
&Z &
}}
\end{equation*}
By \Cref{Rem:iso-comma-Mack-characterization}, we must show that for any $G$-set $T$ the functor 
\begin{align*}
(p,q,\id)_*\colon G\ssetfuz(T, X\times_ZY) &\longrightarrow G\ssetfuz(T,f)/G\ssetfuz(T,g)\\ 
(T\xrightarrow{u}X\times_ZY) &\longmapsto (pu,qu, fpu \overset{\id}{\Rightarrow} gqu)
\end{align*}
is an equivalence of categories.

Let $(t, s,\gamma\colon ft \Rightarrow gs)$ be any object of the target iso-comma category; thus $\gamma\colon T\to G^c$ is a $G$-map such that $\gamma*ft=gs$.
Then for all $x\in T$ we compute
\[
(f\circ (\gamma * t))(x) = f(\gamma(x)\cdot t(x)) = \gamma(x)\cdot f(t(x)) = (\gamma * (f \circ t))(x) \overset{\textrm{hyp.}}{=} (g\circ s)(x)
\]
showing that $f (\gamma*t)= gs$, so that we may define a $G$-map $u\colon T\to X\times_ZY$ into the pullback with components $(\gamma*t,s)$. 
The square of 2-morphisms of $G\ssetfuz$
\[
\xymatrix@L=6pt{
ft
 \ar@{=>}[r]^-{\gamma}
 \ar@{=>}[d]_{f\gamma} & 
gs \ar@{=>}[d]^{\id} \\
fpu \ar@{=>}[r]^-{\id} & 
gqu
}
\]
commutes by the calculation ($x\in T$)
\[
(f\circ \gamma)(x) 
= (\id_f \circ \gamma)(x) 
= \gamma(x) \cdot \id_f(t(x)) 
= \gamma(x)
\]
showing that the pair $(\gamma,\id)$ is a well-defined isomorphism $(t,s,\gamma)\overset{\sim}{\to}(pu,qu,\id)$. This proves the essential surjectivity of $(p,q,\id)_*$.

To show it is fully faithful, consider two $G$-maps $u,v\colon T\to X\times_ZY$. A morphism $(pu,qu,\id)\to (pv,qv,\id)$ between their images is a pair $(\varphi,\psi)$ of 2-cells $\varphi\colon pu \Rightarrow pv$ and $\psi\colon qu \Rightarrow qv$ of $G\ssetfuz$ such that $f\varphi = g\psi$, that is such that
\[
\varphi (x) 
= \varphi(x) \cdot \id_f(pu(x)) 
= (f\circ \varphi)(x)
\overset{\textrm{hyp.}}{=} (g\circ \psi)(x)
= \psi(x) \cdot \id_g( qu(x)) 
= \psi(x)
\]
for all $x\in T$, so that in fact $\varphi$ and $\psi$ must be the same map $T\to G^c$. Moreover, by looking at the pullback components we obtain $\varphi* u= v$:
\begin{align*}
(\varphi*u)(x) 
&= \varphi(x) \cdot u(x) 
= \varphi(x) \cdot \big( pu(x), qu(x)\big)
=  \big(\varphi(x) \cdot  pu(x), \varphi(x) \cdot  qu(x)\big) \\
&= \big( pv (x) ,qv (x) \big)  = v(x)
\end{align*}
In other words, $\varphi$ also defines a 2-cell $u\Rightarrow v$ such that $(p,q,\id)_*(\varphi)=(\varphi, \psi)$, and we see that $(p,q,\id)_*$ induces a bijection on each Hom set as claimed.
 
By \Cref{Rem:theG=Jcase}, it remains to show that $G\ssetfuz$ is an extensive bicategory (\Cref{Def:extensive}).
Indeed, the usual coproducts of $G$-sets are also (strict) coproducts in $G\ssetfuz$, and the 2-functor
\[
G\ssetfuz_{/X} \times G\ssetfuz_{/Y} \longrightarrow G\ssetfuz _{/ X \sqcup Y} \,, \quad (A\overset{a}{\to} X, B \overset{b}{\to} Y) \mapsto (A\sqcup B \xrightarrow{a\sqcup b}X\sqcup Y)
\]
is a biequivalence; both follow easily from the fact that the underlying ordinary category $G\sset$ is extensive (see \Cref{Rem:extensive-cats}), together with the fact that a 2-cell $\varphi\colon (a\sqcup b)\Rightarrow (a'\sqcup b')$ in $G\ssetfuz$, that is a $G$-map $A\sqcup B\to G^c$ such that $\tau*(a\sqcup  b)=a'\sqcup b'$, amounts to the same thing as a pair of 2-cells $\tau_A \colon a \Rightarrow a'$ and $\tau_B\colon b\Rightarrow b'$, that is $G$-maps $\tau_A\colon A\to G^c$ and $\tau_B\colon B\to B^c$ such that $\tau_A*a=a'$ and $\tau_B*b=b'$.
\end{proof}

\begin{Rem}
It is apparent from the proof that the Mackey squares of \Cref{Prop:G-set-fuz} are not strict in general, \ie they are not iso-comma squares. Indeed, it seems that the 2-category $G\ssetfuz$ does not admit any nontrivial iso-comma squares.
\end{Rem}

The two 2-categories we have just defined are actually equivalent:

\begin{Thm} \label{Thm:fusedMackey_vs_gpdGfuz}
For every finite group~$G$, the transport groupoid 2-functor $G\ltimes -$ lifts to a biequivalence of spannable (2,1)-categories
\[
\xymatrix{
G\sset \ar@{->}[d]_{\incl} \ar[r]^-{\sim}_-{\textrm{\textup{Prop.\,\ref{Prop:bieq-Gsets-gpdG}}}}   & \gpdG  \ar@{->}[d]^{\incl}  \\
G\ssetfuz  \ar@{-->}[r]_-\sim^-\exists & \gpdGfuz
}
\]
extending the biequivalence of \Cref{Prop:bieq-Gsets-gpdG} along the two inclusions. 
\end{Thm}

\begin{proof}
The transport groupoid construction $X\mapsto (G\ltimes X,\pi_X)$ of \Cref{Def:transport-gpd}, which gave us the biequivalence $G\sset\isoto\gpdG$, can be extended to a well-defined 2-functor $G\ssetfuz\to \gpdGfuz$ by mapping a 2-cell $\tau\colon f_1\Rightarrow f_2$, that is a $G$-map $\tau\colon X\to G^c$ such that $\tau*f_1=f_2$, to the natural transformation
\[
G\ltimes \tau\colon G\ltimes f_1\Rightarrow G\ltimes f_2
\]
of functors $G\ltimes X\to G\ltimes Y$ whose component at an object $x\in X$ of $G\ltimes X$ is given by the arrow 
$(\tau(x),f_1(x)) \in G\ltimes Y (f_1(x), f_2(x))$.

To see that this $G\ltimes -$ is a well-defined 2-functor it only remains to check that it preserve identity 2-cells and vertical and horizontal composition, all of which is straightforward from the definitions. For horizontal composition this may look a little counter-intuitive, so let us spell it out. 
Consider a horizontal composite \eqref{eq:hor-comp} in $G\ssetfuz$. 
By applying $G\ltimes-$ to the right-hand side, we get the natural transformation $G\ltimes (f_3f_1)\Rightarrow G\ltimes (f_4f_2)$ with component $\tau(x)\cdot \sigma(f_1(x))\colon f_3f_1(x)\to f_4f_2(x)$ at $x\in X$.
After applying $G\ltimes-$ to the left-hand side, we may form the horizontal composite of $G\ltimes \tau$ and $G\ltimes \sigma$ which by definition is the diagonal of the following commutative square of natural transformations:
\[
\xymatrix@L=6pt{
(G\ltimes f_3)(G\ltimes f_1)
 \ar@{=>}[rr]^-{(G\ltimes \sigma)\, \Id}
 \ar@{=>}[d]_{\Id\, (G\ltimes \tau)} && 
(G\ltimes f_4)(G\ltimes f_1) \ar@{=>}[d]^{\Id\, (G\ltimes \tau)} \\
(G\ltimes f_3)(G\ltimes f_2) \ar@{=>}[rr]^-{(G\ltimes \sigma)\, \Id} && 
(G\ltimes f_4)(G\ltimes f_2)
}
\]
By following the right-then-down path, we obtain the (vertical!) composite
\[ (G\ltimes f_4)(G\ltimes \tau) \circ (G\ltimes \sigma)(G\ltimes f_1) \]
whose component at $x\in X$ is given by the following element of~$G$:
\[ \underbrace{(G\ltimes f_4)(G\ltimes \tau)(x)}_{\tau(x)} \cdot \underbrace{(G\ltimes \sigma)(G\ltimes f_1)(x)}_{\sigma(f_1(x))} \,. \]
We see that the two agree, hence $G\ltimes-$ preserves horizontal composites. 

To verify that this 2-functor is a biequivalence as claimed, it suffices to prove that it yields a bijection on each set of 2-cells, which is an immediate consequence of \Cref{Lem:twists_nat} below. 

By \Cref{Prop:G-set-fuz}, the biequivalence of \Cref{Thm:fusedMackey_vs_gpdGfuz} is in fact a biequivalence of spannable (2,1)-categories.
\end{proof}

\begin{Lem} \label{Lem:twists_nat}
Let $f_i\colon X\to Y$ ($i=1,2$) be two $G$-maps. For every natural transformation $\alpha =\{\alpha_x\}_{x\in X}\colon G\ltimes f_1 \Rightarrow G\ltimes f_2$ in $\groupoid$ we have $\alpha = G\ltimes \tau$ for a unique twisting $G$-map $\tau\colon X\to G^c$ such that $\tau*f_1=f_2$. Explicitly, $\tau$ is determined by setting $\alpha_x=(\tau(x),f_1(x))$ for all $x\in X$.
\end{Lem}

\begin{proof}
This follows by inspecting the definitions, because a natural transformation $\alpha\colon G\ltimes f_1 \Rightarrow G\ltimes f_2$ is precisely a collection of pairs $\{\alpha_x= (\tau(x), f_1(x))\}_{x\in X}$ for elements $\tau(x)\in G$ satisfying $\tau(x)f_1(x)= f_2(x)$ and such that
\[
\xymatrix@C=3em{
f_1(x) \ar[d]_-{\alpha_x} \ar[r]^-{(g,f_1(x))} &
 f_1(gx) \ar[d]^-{\alpha_{gx}} \\
f_2(x) \ar[r]^-{(g,f_2(x))} & f_2(gx)
}
\]
commutes in $G\ltimes X$ for all $g\in G$. The latter means that $\tau(gx)g=g\tau(x)$ for all~$g$, that is $x\mapsto \tau(x)$ is a $G$-map $X\to G^c$.
Moreover, the requirement that $f_2(x)= \tau(x)f_1(x)= (\tau*f_1)(x)$ for all $x\in X$ means that $f_2= \tau*f_1$. Hence $\tau\colon f_1\Rightarrow f_2$ is a 2-cell in~$G\ssetfuz$ and $\alpha=G\ltimes \tau$ by the above.
\end{proof}

\begin{Cor} \label{Cor:fusedMackey_vs_gpdGfuz}
By applying the span construction $\spancat_\kk(-)$ of \Cref{Cons:1-spansk} to \Cref{Thm:fusedMackey_vs_gpdGfuz}, we obtain a commutative square of $\kk$-linear categories
\[
\vcenter{
\xymatrix{
\spancat_\kk(G\sset) \ar[d] \ar[r]^-\simeq \ar[d] \ar[dr]|{\spancat_\kk(G\ltimes-)} & \spancat_\kk(\gpdG)  \ar[d]\\
\spancat_\kk(G\ssetfuz) \ar[r]^-{\simeq}  & \spancat_\kk(\gpdGfuz) 
}}
\]
with horizontal equivalences. 
\qed
\end{Cor}

\begin{Cor} \label{Cor:fusedMackeyfun_vs_gpdGfuz-Mackeyfun}
By taking functor categories in \Cref{Cor:fusedMackey_vs_gpdGfuz}, we obtain the following commutative square of $\kk$-linear categories of Mackey functors (\Cref{Def:Mackey-fun})
\[
\vcenter{
\xymatrix{
\Mackey_\kk(G)  & \Mackey_\kk(\gpdG)  \ar[l]_-{\simeq} \\
\Mackey_\kk(G\ssetfuz) \ar[u] & \Mackey_\kk(\gpdGfuz) \ar[u] \ar[l]_-{\simeq} \ar[ul]_{=:\,\Phi}
}}
\]
with horizontal equivalences. Moreover, the full image in $\Mackey_\kk(G)$ of the diagonal functor $\Phi$ coincides with the abelian subcategory $\Mackey^f_\kk(G)$ of \emph{fused Mackey functors} of~\cite{Bouc15}, also called \emph{conjugation-invariant Mackey functors} by Hambleton-Taylor-Williams~\cite{HambletonTaylorWilliams10}. Concretely, they are precisely those $M\in \Mackey_\kk(G)$ such that for every subgroup $H\leq G$ the centralizer $C_G(H)$ acts trivially (via the conjugation maps) on~$M(G/H)$.
\end{Cor}

In particular, this shows that fused Mackey functors too are a special case of our generalized Mackey functors (take $\GG=\JJ=\gpdGfuz$).

\begin{proof}[Proof of \Cref{Cor:fusedMackeyfun_vs_gpdGfuz-Mackeyfun}]
The identification of the image of $\Phi$ with Bouc's $\Mackey^f_\kk(G)$ is now immediate from the definitions. Indeed, the latter is defined to be the (image in $\Mackey_\kk(G)$ of the) $\kk$-linear category of spans on $G\ssetfused=\pih (G\ssetfuz)$ with composition induced by the pullbacks of $G\sset$;  by \Cref{Prop:G-set-fuz} and \Cref{Rem:alt_descr_Span} this is precisely the category $\spancat_\kk(G\ssetfuz)$.

Finally, the explicit characterization of the Mackey functors in $\Mackey^f_\kk(G)$ follows, via \Cref{Lem:twists_nat}, from the fact that an automorphism $\alpha\colon i\Rightarrow i$ in $\gpd$ of a subgroup inclusion homomorphism $i\colon H\hookrightarrow G$ is given by (\ie has for its unique component) an element of $G$ which centralizes~$H$. More precisely, such an element $a\in C_G(H)$ of the centralizer defines a $G$-map $\tau_a\colon G/H\to G^c$ via $\tau_a(gH)=gag^{-1}$, hence a 2-cell $\tau_a\colon \Id_{G/H} \Rightarrow \tau_a*\Id_{G/H}$, hence an equivalence of spans in $G\ssetfuz$:
\[
\vcenter{
\xymatrix@R=.5em@C=4em{ 
& G/H \ar@{=}[dl] \ar@{=}[dr] \ar@{=}[dd] & \\
G/H   && G/H \ar@{}[l]|{\SWcell\tau_a\;\;\;\;\;} \\
& G/H \ar@{=}[ul] \ar[ur]_{\tau_a*\Id} &
}}
\]
The map $\tau_a*\Id_{G/H}$ sends $gH$ to $gaH$, \ie it is precisely the $G$-isomorphism $G/H\overset{\sim}{\to} G/H$ of conjugation by~$a$ (or its inverse, depending on conventions). 
This shows that any Mackey functor $M$ factoring through the quotient $\spancat_\kk(G)\to \spancat_\kk(G\ssetfuz)$ has the property that $C_G(H)$ acts trivially on $M(G/H)$ for all $H\leq G$.

The latter condition is also sufficient for a Mackey functor to factor via the quotient because, as one sees easily, \emph{all 2-cells} of $G\ssetfuz$ are generated by such $\tau_a$ by taking sums (of orbits~$G/H$) and composites (of spans).
(See \cite[Thm.\,2.11]{Bouc15} for more details on this.)
\end{proof}

\begin{Rem} \label{Rem:biset(G)}
The motivation for studying fused Mackey functors is that they are precisely the Mackey functors for~$G$ which can be formulated as biset functors (\cf \Cref{sec:bisets}). More precisely, the authors of \cite{HambletonTaylorWilliams10} consider the (non-full) $\kk$-linear subcategory $\bisetcat_\kk(G) \subset \bisetcatbif(\group)$ (see \Cref{Not:free-bifree}) whose objects are the finite subgroups of~$G$ and whose maps are the \emph{conjugation bisets} between them; the latter are all bisets which can be obtained by combining restriction and induction bisets (see \Cref{Rem:elementary_bisets}) as well as those isomorphism bisets given by conjugation by an element of~$G$. 
Bouc \cite{Bouc15} constructs a $\kk$-linear functor $\spancat_\kk(G\sset)\to \bisetcat_\kk^\oplus(G)$ to the additive hull $\bisetcat_\kk^\oplus(G)$ of $\bisetcat_\kk(G)$ which sends induction, restriction and conjugation spans to the homonymous bisets, and proves that it descends to an equivalence $\spancat_\kk(G\ssetfuz) \overset{\sim}{\to} \bisetcat_\kk^\oplus(G)$ on spans of fused $G$-sets  (\cite[Thm.\,2.11]{Bouc15}). 
Hence, fused Mackey functors are indeed expressible as a kind of biset functors. 

We suspect the latter equivalence is the 1-truncation of a biequivalence between $G\ssetfuz \simeq \gpdGfuz$ and a suitable comma bicategory of bisets over~$G$, which can be obtained as a variant of the realization pseudo-functor $\cat R$ of \Cref{Thm:Huglo_bifunctor} below. We do not pursue this idea here, as it would take us a little afield, but the next section should provide most of the ingredients to do so.
\end{Rem}

\section{Spans vs bisets}
\label{sec:bisets}%

In this section we investigate the relationship between Mackey functors and biset functors. 
Recall from Examples~\ref{Def:biset_functor_gpd} and~\ref{Exa:truncated-bisets} the bicategory $\biset(\gpd)$ of finite groupoids, bisets and biset morphisms, as well as its full sub-bicategory $\biset(\group)$ of finite groups. Consider their 1-truncations $\pih (\biset(\gpd))$ and $\pih (\biset(\group))$.

\begin{Lem} \label{Lem:biset-add}
The category $\pih \biset(\gpd)$ is semi-additive (\S\,\ref{parag:k-lin}), with direct sums induced by the disjoint sums of groupoids. 
The sum of two bisets $H\to G$ and the zero biset $H\to G$ are given by the coproduct biset and the constantly empty biset 
\[
[{}_GU_H] + [{}_GV_H] = [U \sqcup V]  
\quad \textrm{ and } \quad
0_{H,G}=[ \varnothing ]
\]
respectively. 
Moreover, $\pih \biset(\gpd)$ is the semi-additive hull of its full subcategory $\pih \biset(\group)$, meaning that every object of $\pih \biset(\gpd)$ is a direct sum of finite groups and every arrow is a matrix of arrows between groups.
\end{Lem}

\begin{proof} This is all straightforward. \end{proof}

\begin{Not} \label{Not:bisetcats}
As usual, fix a ground commutative ring~$\kk$.
We will write 
\[
\bisetcat(\gpd) := \pih \biset(\gpd)
\quad \textrm{ and } \quad
\bisetcat(\group) := \pih \biset(\group)
\]
for the semi-additive categories of bisets of \Cref{Lem:biset-add}, and we will denote by
\[
\bisetcat_\kk(\gpd) := \kk (\pih \biset(\gpd)) 
\quad \textrm{ and } \quad
\bisetcat_\kk(\group) := \kk (\pih \biset(\group))
\]
their $\kk$-linearization as in \Cref{Cons:kk-linearization}. 
\end{Not}

\begin{Rem}
In the literature, the category $\bisetcat_\kk(\group)$ is the one usually referred to as `the biset category', rather than its additive hull $\bisetcat_\kk (\gpd)$.
In \cite{IbarraRaggiCardenasRomero18}, the authors provide an alternative explicit description of the additive completion, which avoids the use of groupoids and is related to the 2-category~$\mathbb S$ of \cite{Nakaoka16}.
\end{Rem}

\begin{Def}[{Biset functor~\cite{Bouc10}}]
\label{Def:biset-functor}
A \emph{\textup($\kk$-linear\textup)  biset functor} is a $\kk$-linear (hence also additive, \ie direct-sum preserving) functor
\[
F\colon \bisetcat_\kk(\gpd) \longrightarrow \kk\MMod
\]
to the abelian category of $\kk$-modules. A morphism of biset functors is simply a natural transformation.
\end{Def}

\begin{Rem} \label{Rem:various-cats-of-biset-functors}
By the general properties of $\kk$-linearization, we have a canonical additive functor $\bisetcat(\gpd) \to \bisetcat_\kk(\gpd)$ inducing a $\kk$-linear isomorphism
\begin{equation*} 
\Fun_\kk(\bisetcat_\kk(\gpd), \kk\MMod ) \stackrel{\sim}{\longrightarrow} \Fun_+(\bisetcat(\gpd), \kk\MMod ) 
\end{equation*}
of functor categories. Moreover, since additive functors extend essentially uniquely to the additive hull, we also have an equivalence
\begin{equation*} 
\Fun_\kk(\bisetcat_\kk(\gpd), \kk\MMod ) \stackrel{\sim}{\longrightarrow} \Fun_\kk(\bisetcat_\kk(\group), \kk\MMod ) = :\cat F
\end{equation*}
induced by the inclusion $\bisetcat_\kk(\group)\hookrightarrow \bisetcat_\kk(\gpd)$. This last functor category, under the notation~$\cat F$, is the category of biset functors as defined in~\cite{Bouc10}.
\end{Rem}

\begin{Rem}
A more combinatorial description of biset functors, similar to Green's original axioms for Green and Mackey functors~\cite{Green71}, is given in~\cite[\S8]{Webb00} (under the name \emph{globally defined Mackey functors}). See also \Cref{Rem:Webb's-variations}. 
\end{Rem}

The following result provides a direct connection between bisets and spans:

\begin{Thm}[{Huglo \cite{Huglo19pp}}]
\label{Thm:Huglo_bifunctor}
There is a well-defined pseudo-functor
\[
\cat R\colon \Span(\groupoid) \longrightarrow \biset(\groupoid)
\]
from the bicategory of spans in $\gpd$ (\Cref{Rem:span-bicats}) to that of bisets (\Cref{Def:biset_functor_gpd}),  which is the identity on objects, and which sends a span $H \stackrel{b}{\gets} S \stackrel{a}{\to} G$ of groupoids to the $G,H$-biset
\[
\cat R(H \stackrel{b}{\gets} S \stackrel{a}{\to} G)
:= G(a-,-) \otimes_S H(-,b-) \colon H^\op\times G \too \Set \,.
\]
Moreover, the functor it induces on truncated 1-categories
\[
\Real:= \pih \cat R \colon \spancat(\groupoid) \longrightarrow \bisetcat(\groupoid)
\]
is additive and full. 
\end{Thm}

For the present purposes, we really only need the following corollary which is immediately obtained by $\kk$-linearization:

\begin{Cor} \label{Cor:kklinear-Re}
There exists a full $\kk$-linear functor
\[
\Real_\kk := \kk (\pih \cat R) \colon \spancat_\kk (\gpd) \longrightarrow \bisetcat_\kk(\gpd)
\]
which is the identity on objects and sends the class of a span  $[H \stackrel{b}{\gets} S \stackrel{a}{\to} G]$ to the class of the tensor product $G,H$-biset $[G(a-,-)\otimes_S H(-,b-)]$. 
\qed
\end{Cor}

\begin{proof}[Sketch of proof for \Cref{Thm:Huglo_bifunctor}]
More details can be found in Huglo's PhD thesis \cite{Huglo19pp}, together with other properties of the pseudo-functor~$\cat R$; see also  \cite[\S5]{DellAmbrogioHuglo20pp} for a compact account. The main observation is that for every functor $u\colon H\to G$ there is an internal adjunction (\S\,\ref{parag:internal-adj-equiv}) in the bicategory $\biset(\groupoid)$ as follows:
\[
\xymatrix{
H
\ar@/_2ex/[d]_{\cat R_!(u)\,:=\,G(u-,-)}
\ar@/^2ex/@{<-}[d]^{G(-,u-)\,=:\,\cat R^*(u)}
 \ar@{}[d]|{\dashv} \\
G
}
\]
Routine properties of adjunctions allow us to extend the collection of the left adjoints to a pseudo-functor $\cat R_!\colon \gpd\to \biset$, and similarly the right adjoints to a pseudo-functor $\cat R^*\colon \gpd^{\op}\to \biset$, both pseudo-functors being the identity on objects. By an explicit computation, one verifies that these adjunctions satisfy the base-change formula with respect to iso-comma squares. It follows then by the universal property of the span bicategory (see \cite[Thm.\,5.2.1]{BalmerDellAmbrogio18pp}) that the pseudo-functors $\cat R_!$ and $\cat R^*$ can be `pasted together' in order to obtain a pseudo-functor~$\cat R$ on $\Span(\gpd)$ which sends a span $H \stackrel{b}{\gets} S \stackrel{a}{\to} G$ to the horizontal composite $\cat R_!(a) \circ \cat R^*(b)$ in~$\biset(\gpd)$, as claimed.
\[
\xymatrix@R=14pt{
\groupoid
 \ar[d]_{\incl}
 \ar@/^2ex/[drr]^-{\cat R_!} && \\
\Span(\groupoid) \ar@{-->}[rr]^-{\exists\,!\;\cat R} && \biset(\groupoid)
\\
{\groupoid^{\op}}
 \ar[u]^{\incl}
 \ar@/_2ex/[urr]_-{\cat R^*} && \\
}
\]
The additivity of $\pih \cat R$ boils down to the fact that $\cat R$ is the identity on objects and preserves disjoint unions.

In order to see that the functor $\pih{\cat R}$ is full, it can be shown that every biset ${}_GU_H$ is isomorphic to the image under $\cat R$ of a canonical span $H \gets S_U \to G$, where the groupoid $S_U$ is a suitable Grothendieck construction (category of elements).
Alternatively, and more simply, it suffices to combine the additivity of $\pih \cat R$ with Remarks~\ref{Rem:elementary_bisets} and~\ref{Rem:comparison-of-elementary-generators} below.
\end{proof}

We are going to upgrade \Cref{Cor:kklinear-Re} to the following more precise result:

\begin{Thm} \label{Thm:biset-as-span-quot}
The realization pseudo-functor of \Cref{Thm:Huglo_bifunctor} induces 
an isomorphism of $\kk$-linear categories 
\[
\xymatrix@C=5pt@R=10pt{
\spancat_\kk (\gpd)
 \ar@{->>}[dr] \ar[rr]^-{\Real_\kk} &&
 \bisetcat_\kk(\gpd) \\
& \spancat_\kk (\gpd) /_\sim \ar[ru]_{\simeq} &
}
\]
which identifies the biset category with the quotient of the full span category of \Cref{sec:presentation} obtained by factoring out the additive $\kk$-linear ideal generated by the relations 
\begin{align*}
[Q \overset{\;p}{\leftarrow} G \overset{p\;}{\rightarrow} Q] \sim \id_{Q}
\end{align*}
for all surjective group homomorphisms $p\colon G\to Q$. Or equivalently  (in terms of the elementary spans of \Cref{Not:generators}), generated by the `deflativity relations'
\begin{align} \label{eq:deflative_equation}
\Defl^G_{G/N} \circ \Infl^G_{G/N} \sim \id_{G/N}\end{align}
for every normal subgroup $N\unlhd G$.
\end{Thm}

Before proving the theorem, let us immediately record:

\begin{Cor} 
\label{Cor:Nakaoka-embedding}%
Precomposition with the functor $\Real_\kk$ of \Cref{Cor:kklinear-Re} induces a fully faithful embedding of functor categories
\[
\Real_\kk^*\colon
\Fun_\kk (\bisetcat_\kk(\groupoid), \kk\MMod)
\hookrightarrow
\Fun_\kk (\spancat_\kk(\groupoid) , \kk\MMod)
\]
which identifies biset functors (\Cref{Def:biset-functor}) as a full reflexive $\kk$-linear subcategory of global Mackey functors on~$\gpd$ (\Cref{Exa:gpd}).
The essential image of the embedding consists precisely of the Mackey functors  $M$ with the property that $M([Q\leftarrow G \rightarrow Q]) = \id_{M(Q)}$  for every surjective group homomorphism $G\to Q$.
\end{Cor}

\begin{proof}
The characterization of the image is immediate from \Cref{Thm:biset-as-span-quot}.

The rest follows by standard arguments.
Using that $\Real_\kk$ is full and surjective on objects, one verifies immediately that it induces a fully faithful functor between functor categories.
Its image is a reflective subcategory, \ie the inclusion functor admits a left adjoint. In fact, this left adjoint  (the `reflection') is provided (after composing with~$\Real_\kk^*$) by the $\kk$-linear left Kan extension along~$\Real_\kk$, \ie by the unique colimit-preserving functor sending a representable Mackey functor $\spancat_\kk(\groupoid)(G, -)$ to the corresponding representable biset functor~$\bisetcat_\kk(\groupoid)(G,-)$. 
\end{proof}

\begin{Rem} \label{Rem:Nakaoka-on-bisets-vs-spans}
\Cref{Cor:Nakaoka-embedding} contains the main result of \cite{Nakaoka16}. Nakaoka calls \eqref{eq:deflative_equation} the \emph{deflativity condition} and the Mackey functors arising this way from biset functors \emph{deflative Mackey functors}.
\end{Rem}

In order to prove \Cref{Thm:biset-as-span-quot}, it will suffice to compare suitable presentations of the two categories. In fact, the realization pseudo-functor $\cat R$ of \Cref{Thm:Huglo_bifunctor} is not really necessary for the proof: once the two presentations are established, one can see that the functor of \Cref{Cor:kklinear-Re} must exist for formal reasons. However, it is nice to know that the comparison of spans and bisets comes from such a natural construction as~$\cat R$, whose 2-categorical nature fits well in this article's philosophy. 

\begin{Rem}[A presentation of the biset category] \label{Rem:elementary_bisets}
We recall from \cite[\S2.3]{Bouc10} that every $G,H$-biset between finite groups decomposes as a coproduct of \emph{transitive} $G,H$-bisets (\ie those which are indecomposable with respect to coproducts), and every transitive biset ${}_GX_H\colon H\to G$ is isomorphic to a unique horizontal composite of the form:
\begin{equation} \label{eq:can-form-bisets}
\xymatrix{
G &
D \ar[l]_-{\Ind_D^G} &
D/C \ar[l]_-{\Infl_{D/C}^D} &
B/A \ar[l]_-{\Iso(f)}^-{\sim} &
B \ar[l]_-{\Defl^B_{B/A}} &
H \ar[l]_-{\Res^H_B}
}
\end{equation}
Here $A\unlhd B\leq H$ and $C\unlhd D\leq G$ are subgroups, with $A$ and $C$ normal in $B$ and~$D$, respectively, $f\colon B/A\stackrel{\sim}{\to}D/C$ is an isomorphism of groups, and the notations refer to the following five kinds of \emph{elementary bisets} (\cite[\S2.3.9]{Bouc10}):
\begin{itemize}
\item \emph{Isomorphisms:} $\Iso(f):= {}_GG_{G'}$, with $G$ acting on itself on the left and $G'$ acting on $G$ by right multiplication via a group isomorphism $f\colon G'\stackrel{\sim}{\to} G$. 
\item \emph{Restrictions:} $\Res^G_H:= {}_HG_G$, defined whenever $H$ is a subgroup of~$G$.
\item \emph{Inductions:} $\Ind_H^G:= {}_GG_H$, again for $H$ a subgroup of~$G$.
\item \emph{Deflations:} $\Defl^G_{G/N}:= {}_{G/N}(G/N)_G$, for any normal subgroup $N$ of~$G$.
\item \emph{Inflations:} $\Infl_{G/N}^G:= {}_{G}(G/N)_{G/N}$, again for a normal subgroup $N$ of~$G$.
\end{itemize}
In particular, the morphisms of the $\kk$-linear category $\bisetcat_\kk(\group)$ are generated by the (isomorphism classes of) elementary bisets.
As already mentioned in \Cref{Rem:comparison-biset}, Bouc also provides a full list of relations for the elementary bisets, and therefore a presentation of $\bisetcat_\kk(\group)$ as a $\kk$-linear category (cf.\ \cite[\S\,3.1]{Bouc10}). Bouc's list is nearly identical with that in \Cref{Thm:pres-Span}, the only difference being that Bouc's relation set 2.(d) does not require the condition ``$M\cap N=1$''.
\end{Rem}

\begin{Rem} \label{Rem:comparison-of-elementary-generators}
It is immediate to verify that the functor $\Real_\kk$ maps each elementary span of \Cref{Not:generators} to the homonymous elementary biset of \Cref{Rem:elementary_bisets}. Just remember that, in our conventions, a span $[G \leftarrow S \rightarrow H]$ is read left-to-right while a biset ${}_HU_G$ is read right-to-left, \ie they both stand for a map from $G$ to~$H$.
\end{Rem}

\begin{proof}[Proof of \Cref{Thm:biset-as-span-quot}]
In view of \Cref{Cor:kklinear-Re}, it only remains to identify the kernel on morphisms of the $\kk$-linear functor $\Real_\kk\colon \spancat_\kk(\gpd)\to \bisetcat_\kk(\gpd)$.
In fact, since the two inclusions $\spancat_\kk(\group) \subset \spancat_\kk(\gpd)$ and $\bisetcat_\kk(\group) \subset \bisetcat_\kk(\gpd)$ are additive hulls (by \Cref{Lem:add-closure-spans} and \Cref{Lem:biset-add} respectively), we only need to look at the restriction of $\Real_\kk$ to a functor $\spancat_\kk(\group)\to \bisetcat_\kk(\group)$ between groups.

Let us begin by noting that, since in $\spancat_\kk(\group)$ we have the equality
\[
[G = G \stackrel{p\;}{\rightarrow} Q] \circ [Q \stackrel{\;p}{\leftarrow} G= G] = [Q \stackrel{\;p}{\leftarrow} G \stackrel{p\;}{\rightarrow} Q]
\]
and since every surjective homomorphism $G \stackrel{p\;}{\to} Q$ is isomorphic to one of the form $G\to G/N$,  the relations \eqref{eq:deflative_equation} are evidently equivalent to the family of relations
\begin{align*}
[Q \overset{\;p}{\leftarrow} G \overset{p\;}{\rightarrow} Q] \sim \id_{Q} 
\end{align*}
for all surjective group homomorphisms $p\colon G\to Q$, as stated.

Now we compare the presentation of $\spancat_\kk(\group)$ given in \Cref{Thm:pres-Span} and Bouc's presentation of $\bisetcat_\kk(\group)$ recalled in \Cref{Rem:elementary_bisets}.

As already mentioned in \Cref{Rem:comparison-of-elementary-generators}, the functor $\Real_\kk$ maps generators (the elementary spans) to the homonymous generators (the elementary bisets), and the only difference between the two presentations lies within the relation family~2.(d). 
Hence it only remains to prove that the relations \eqref{eq:deflative_equation}, together with all the relations of the span category (as in \Cref{Thm:pres-Span}), imply the following:

\begin{enumerate}[\rm1.]
\item[] \emph{Bouc's 2.(d):}
For any two normal subgroups $M,N\trianglelefteq G$ we have
\[
\Defl^G_{G/N} \circ \Infl^G_{G/M} = \Infl^{G/N}_{G/NM} \circ \Defl^{G/M}_{G/NM}  \,.
\]
\end{enumerate} 
 
To this end, let $M,N$ be any two normal subgroups of a group~$G$, and consider the following diagram of group homomorphisms, where the outer square consists of the quotient maps and the inner square is a pullback:
\[
\xymatrix@C=14pt@R=16pt{
&  G \ar[d]_p \ar@/_3ex/@{->>}[ddl]_{\overline b} \ar@/^3ex/@{->>}[ddr]^{\overline a} & \\
& P \ar@{->>}[dr]^{\tilde a} \ar@{->>}[dl]_{\tilde b} & \\
G/M   \ar@{->>}[dr]_a && G/N \ar@{->>}[dl]^b \\
& G/MN&
}
\]
One verifies easily that the comparison map $p\colon G\to P$ is also surjective, hence by \eqref{eq:deflative_equation} we get the relation $p_*p^* = \id_P$; here and below we use the short-hand notations $(-)_*$ and $(-)^*$ as in \Cref{Rem:*notation}. Now we compute as follows, using also the basic functoriality of $(-)_*$ and $(-)^*$ afforded by the relations 0.(a) and~1.(c):
\begin{align*}
b^*a_* 
 &\;=\; \tilde a_* \tilde b^* && \textrm{by \Cref{Rem:pullbacks-and-2d} and 2.(d) for spans} \\
 &\;=\; \tilde a_* \,\id_P\, \tilde b^* && \\
 &\;=\; \tilde a_* p_*p^* \tilde b^* && \textrm{by \eqref{eq:deflative_equation} for }p \\
 &\;=\; (\tilde a p)_* (\tilde b p)^* && \\
 &\;=\; (\overline a)_* (\overline b)^* &&
\end{align*}
This proves Bouc's relation~2.(d), and concludes the proof of the theorem.
\end{proof}

\begin{Rem} \label{Lem:Bouc_simplified}
The argument of the above proof shows in fact that, in the presence of his other relation families, Bouc's relations of type 2.(d) for bisets, \ie the relations
\[ \Defl^G_{G/N} \circ \Infl^G_{G/M} = \Infl^{G/N}_{G/NM} \circ \Defl^{G/M}_{G/NM} \quad \quad \textrm{ for all } N,M \unlhd G \,,\]
follow already from the special case with $N=M$.
\end{Rem}

\begin{center}$***$\end{center}

In the remainder of this last section we consider a couple of frequently used variants of biset functors and show that the analogue of \Cref{Thm:biset-as-span-quot} provides for each of them an \emph{equivalence} with the corresponding notion of Mackey functors.

\begin{Ter} \label{Ter:free_bifree}
Recall that (for groups $G$ and $H$) a $G,H$-biset ${}_GU_H$ is \emph{right-free} if $H$ acts freely on~$U$, \emph{left-free} if $G$ acts freely, and \emph{bifree} if both actions are free. Of the five elementary kinds of bisets in Remark~\ref{Rem:elementary_bisets}, we see that isomorphisms, restrictions and inductions are bifree; inflations are only right-free; and deflations are only left-free (unless of course, for the latter two kinds, we are in the degenerate case $N=1$).
Both right-free and bifree bisets form classes closed under horizontal composition, so we can consider the 2-full sub-bicategories containing them.
\end{Ter}

\begin{Not}
\label{Not:free-bifree}%
In a way which is hopefully self-explanatory, we may write
\[
\biset^\rfree\!(\group), \quad
\biset^\bifree\!(\group) \quad\quad \textrm{and} \quad\quad
\biset^\rfree\!(\groupoid), \quad
\biset^\bifree\!(\groupoid)
\]
for the 2-full sub-bicategories of $\biset(\group)$ and $\biset(\groupoid)$, respectively, where only right-free (``$\rfree$''), or bifree (``$\bifree$''), bisets are allowed as 1-cells. (These notations are used at the end of the Introduction.) We will be interested in their $\kk$-linearized 1-truncations, for which we use the following notations:
\[
\bisetcatrf(\group), \quad
\bisetcatbif(\group) \quad\quad \textrm{and} \quad\quad
\bisetcatrf(\groupoid), \quad
\bisetcatbif(\groupoid) \,.
\]
\end{Not}

\begin{Rem} \label{Rem:biset-var-generation}
Clearly, both $\bisetcatrf(\group)\subset \bisetcatrf(\groupoid)$ and $\bisetcatbif(\group) \subset \bisetcatbif(\groupoid)$  are again the inclusion of a $\kk$-linear category in its additive hull. By construction, moreover, they are generated as $\kk$-linear categories, respectively as additive $\kk$-linear categories, by the following elementary bisets: 
\begin{itemize}
\item $\bisetcatrf(\group)$ and $\bisetcatrf(\groupoid)$: by isomorphisms, restrictions, inductions and inflations (\ie deflations are not allowed).
\item $\bisetcatbif(\group)$ and $\bisetcatbif(\groupoid)$: by isomorphisms, restrictions and inductions (\ie neither inflations nor deflations are allowed).
\end{itemize}
\end{Rem}

\begin{Rem} \label{Rem:Webb's-variations}
For even more variation, we may follow \cite[\S\,8]{Webb00} and choose three classes of finite groups $\cat D, \cat X, \cat Y$, with $\cat X$ and $\cat Y$ closed under the formation of group extensions and subquotients. Such a triple defines a full subcategory of bisets where the objects are the groups in $\cat D$, and where the morphisms are those bisets whose right and left isotropy groups belong to $\cat X$ and $\cat Y$, respectively. In other words, the morphisms include all inductions, restrictions and isomorphisms between the available groups of~$\cat D$, but only the inflations along quotient homomorphisms with kernel in $\cat X$ and only deflations for those with kernel in~$\cat Y$.
For instance, right-free bisets correspond to choosing $\cat X=\{1\}$ and $\cat Y=\{\textrm{all groups}\}$ and bifree bisets to $\cat X=\cat Y=\{1\}$.

For each triple $(\cat D,\cat X,\cat Y)$, the $\kk$-linear functors on the associated biset category is equivalent to Webb's category $\Mackey_\kk^{\cat X,\cat Y}\!(\cat D)$ of \emph{globally defined Mackey functors}.
Thus  the latter can be identified with a particular kind of biset functors. Hence, by the arguments of this section, they can be seen as a (full subcategory of a) kind of generalized Mackey functors. We leave the details to the interested reader and only treat here the above two chosen special cases.
\end{Rem}

\begin{Thm} \label{Thm:no-deflations}
The realization pseudo-functor of \Cref{Thm:Huglo_bifunctor} induces 
two isomorphisms of $\kk$-linear categories 
\[
\spancat_\kk(\groupoid;\groupoidf)
\overset{\sim}{\too}
\bisetcatrf(\groupoid)
\quad \textrm{ and } \quad
\spancat_\kk(\groupoidf; \groupoidf)
\overset{\sim}{\too}
\bisetcatbif(\groupoid)
\]
which identify the right-free biset category and the bifree biset category with suitable categories of spans of groupoids (see Examples~\ref{Exa:gpdf} and~\ref{Exa:gpd-gpdf}).
\end{Thm}

Therefore, once again, we can subsume the corresponding notion of functors under our generalized Mackey functors:

\begin{Cor} 
\label{Cor:no-deflations}
Precomposition with the functor of \Cref{Thm:no-deflations} induces equivalences (in fact isomorphisms) of functor categories
\[
\Fun_\kk (\bisetcatrf(\groupoid), \kk\MMod)
\overset{\sim}{\too}
\Fun_\kk (\spancat_\kk(\groupoid;\groupoidf) , \kk\MMod)
\]
and
\[
\Fun_\kk (\bisetcatbif(\groupoid), \kk\MMod)
\overset{\sim}{\too}
\Fun_\kk (\spancat_\kk(\groupoidf; \groupoidf) , \kk\MMod)
\]
which identify right-free and bifree biset functors with certain categories of generalized Mackey functors.
\qed
\end{Cor}

\begin{proof}[Proof of \Cref{Thm:no-deflations}]
It is an immediate consequence of Remarks~\ref{Rem:elementary_bisets} and~\ref{Rem:comparison-of-elementary-generators} that the realization pseudo-functor $\cat R$ of Theorem~\ref{Thm:Huglo_bifunctor} restricts to pseudo-functors
\[
\Span(\groupoid;\groupoidf) \stackrel{\cat R}{\longrightarrow} \biset^\rfree\!(\groupoid)
\quad \textrm{ and } \quad
\Span(\groupoidf; \groupoidf) \stackrel{\cat R}{\longrightarrow} \biset^\bifree\!(\groupoid) \,.
\]
Indeed, on the side of spans, limiting right (resp.\ right and left) legs to faithful functors precisely eliminates the elementary spans $G = G \onto G/N$ (resp.\ both $G = G \onto G/N$ and $G/N \twoheadleftarrow G = G$) from the set of generators. 

The two 1-truncated functors $\pih \cat R$ are full, like the one of \Cref{Thm:Huglo_bifunctor}, and so are the induced $\kk$-linear functors
\begin{equation} \label{eq:wanna-be-isos-rf-bif}
\spancat_\kk(\groupoid;\groupoidf)
\too
\bisetcatrf(\groupoid)
\quad \textrm{ and } \quad
\spancat_\kk(\groupoidf; \groupoidf)
\too
\bisetcatbif(\groupoid)
\end{equation}
both of which will still be denoted by $\Real_\kk$.
It only remains to see that these two functors are faithful, and we will do so by comparing two presentations, as in the proof of \Cref{Thm:biset-as-span-quot}.

To this end, we make the following claim: For each of the four $\kk$-linear categories involved in~\eqref{eq:wanna-be-isos-rf-bif}, we can obtain a presentation simply by `restricting' the presentation of \Cref{Thm:pres-Span} (for the two categories of spans) or \Cref{Rem:elementary_bisets} (for the two categories of bisets). In other words, it suffices to ignore the irrelevant generators and relations and keep the rest. Thus, explicitly, the categories $\spancat_\kk(\groupoid;\groupoidf)$ and $\bisetcatrf(\groupoid)$ are generated by the elementary spans (resp.\ bisets) other than the deflations, and all relations follow from those of the families 0.-2.\ other than 1.(c), 2.(b), 2.(d) and some of the ones in 0.(a) and 2.(e-f) (because they involve deflations). Similarly, $\spancat_\kk(\groupoidf; \groupoidf)$ and $\bisetcatbif(\groupoid)$ are generated by the elementary spans/bisets other than inflations and deflations, with relations determined by the relations 0.-2.\ other than 1.(c), 2.(b), 2.(d-f) and some of those in 0.(a).
Since in both cases we are led to ignore the relations of type 2.(d), the corresponding presentations for spans and bisets now look identical, from which it follows that the two functors $\Real_\kk$ of~\eqref{eq:wanna-be-isos-rf-bif} are indeed $\kk$-linear isomorphisms.

To see why the above claim on presentations is true, first notice that in each of the four categories the retained generators \emph{do} generate, by construction. Moreover, it is immediate to see that in each case the retained relations still allow us to commute or fuse any pair of the remaining generators, hence they still let us reduce an arbitrary string of generators to a linearly independent finite sum of short strings in the length-six canonical form of \eqref{eq:can-form}
\[
\Ind^G_B \circ \Iso(f) \circ \Defl^S_{S/M} \circ \Infl^S_{S/N} \circ \Iso(\ell^{-1}) \circ \Res^H_D \quad \quad \textrm{ (for spans)}
\]
respectively in the length-five canonical form of \eqref{eq:can-form-bisets} 
\[
\Ind_D^G \circ \Infl_{D/C}^D \circ \Iso(f) \circ \Defl^B_{B/A} \circ \Res^H_B
\quad\quad \textrm{(for bisets)}
\]
(of course, deflations resp.\ deflations and inflations are now absent from both forms).
For bisets, the length-five canonical form of a string (or of an $\sqcup$-irreducible biset) is unique, hence any set of relations that allows us to bring each string of generators to its canonical form is sufficient to determine all relations that hold within the ambient category $\bisetcat_\kk(\groupoid)$. 

For spans, the argument is slightly subtler. First  notice that we may consider $\spancat_\kk(\gpd; \groupoidf)$ and $\spancat_\kk(\groupoidf; \groupoidf)$ to be (non-full) subcategories of $\spancat_\kk(\gpd)$, because the morphisms of pairs $(\gpd; \groupoidf)\to (\gpd; \gpd)$ and $(\groupoidf; \groupoidf)\to (\gpd;\gpd)$ induce faithful functors on span categories. This is because $\groupoidf\subset \gpd$ is a 2-full 2-subcategory containing all equivalences, hence the data of any equivalences between spans with one or two faithful legs is already available in $(\gpd; \groupoidf)$, resp.\ in $(\groupoidf; \groupoidf)$.
Now in the ambient category $\spancat_\kk(\gpd)$, as we have seen in the proof of \Cref{Thm:pres-Span}, the above length-six canonical form of a string (or of an $\sqcup$-irreducible span) is only unique up to changing $(D,S,N,M,B,f,\ell)$ by a bunch of compatible isomorphisms, as in~\eqref{eq:standard-form-almost-uniqueness}; but all elementary isomorphisms are still available in the two subcategories, as well as the commutativity relations between isomorphisms and the other retained generators. Thus, once again, we conclude that the retained families of relations allow us to determined all relations between strings as they hold in $\spancat_\kk(\gpd)$.

This concludes the proof.
\end{proof}

\begin{Rem} \label{Rem:Miller}
Miller \cite{Miller17} offers a different approach to \Cref{Thm:no-deflations}, already at the bicategorical level.
More precisely, the main result of \emph{loc.\,cit.\ }is an explicit biequivalence $\mathbb B \simeq \mathbb C$ between the `Burnside bicategory of groupoids'~$\mathbb B$ and the `bicategory of correspondences' ($=$~nice spans)~$\mathbb C$. Infinite groupoids are allowed in $\mathbb B$ and~$\mathbb C$, but if we restrict attention to finite groupoids we easily find biequivalences 
\[ 
(\mathbb B|_{\mathsf{fin}} )^\op\simeq \biset^\rfree\!(\gpd)_\mathsf{core}
\quad \textrm{ and } \quad
(\mathbb C|_{\mathsf{fin}})^\op \simeq \Span(\gpd; \groupoidf)_\mathsf{core} \,.
\]
Here $(\ldots)|_{\mathsf{fin}}$ denotes the 2-full bicategory of finite groupoids and finite bisets, resp.\ of spans of finite groupoids; the decoration $(\ldots)_\mathsf{core}$ indicates that we discard all non-invertible 2-cells in our two bicategories of  \Cref{Not:free-bifree} and \Cref{Rem:span-bicats}. 
\Cref{Thm:no-deflations} and \Cref{Cor:no-deflations} can then be deduced from Miller's result, by arguing as above.
While the first biequivalence above is immediate, the second one involves replacing every span $H \gets P \to G$ with $P\to G$ faithful with an equivalent span $H \gets \tilde P \to G$ with the property that $\tilde P \to G$ is a `weak finite cover' (see \cite[Def.\,4.1]{Miller17}). 
To this end, just replace $i$ by $\tilde i$ as in the iso-comma square
\begin{equation*}
\vcenter{
\xymatrix@C=14pt@R=14pt{
& \tilde P \ar[dl]_\simeq \ar[dr]^-{\tilde i} & \\
P \ar[dr]_i \ar@{}[rr]|{\oEcell{\sim}} && G \ar@{=}[dl] \\
&G &
}}
\end{equation*}
or, alternatively, using the functorial fibrant replacement in the canonical model structure of groupoids (or categories); see \eg \cite{Rezk99}. 
A final subtle point is that the horizontal composition of~$\mathbb C$ is defined using strict pullbacks, but since its spans have a leg which is a fibration, this agrees with our composition via iso-commas. 
\end{Rem}

\begin{Rem} \label{Rem:next-upgrade}
(Added `in proof':) Upgrading \Cref{Thm:no-deflations}, and sharpening the bicategorical image drawn at the end of the Introduction, we now know for a fact that the realization pseudo-functor~$\cat R$ actually induces biequivalences
\[
\Span(\groupoid;\groupoidf) \overset{\sim}{\longrightarrow} \biset^{\rfree}\!(\groupoid)
\quad \textrm{ and } \quad
\Span(\groupoidf) \overset{\sim}{\longrightarrow} \biset^{\bifree}\!(\groupoid) \,,
\]
without any need to first throw away the non-invertible 2-cells as per Miller's result (see \Cref{Rem:Miller}). This will appear and be put to good use in~\cite{BalmerDellAmbrogiopp}. 
\end{Rem}


\printindex
\end{document}